\documentclass[]{article}

\newcommand{\pd}{\partial}
\newcommand{\vu}{\mathbf{u}}

\newcommand{\kk}{\mathbf{k}}
\newcommand{\vx}{\mathbf{x}}
\newcommand{\Lop}{\mathcal{L}}

\newcommand{\yy}{\mathbf{y}}

\newcommand{\ees}{e^{i\Delta T\lambda_{n}s}}
\newcommand{\eets}{e^{i\Delta T\lambda_{n}(t+s)}}
\newcommand{\pp}{\frac{1}{\eta}\int_{0}^{\eta}\rho\left(\frac{s}{\eta}\right)}

\usepackage{color}
\usepackage{float}
\usepackage{graphicx}
\usepackage{amsmath}
\usepackage{amsfonts}
\usepackage{amsthm}
\usepackage{amssymb}
\usepackage{thmtools}
\usepackage{cleveref}

\declaretheorem[style=theorem, qed=$\blacklozenge$]{theorem}
\declaretheorem[style=lemma, qed=$\blacklozenge$]{lemma}

\author{
Peddle, Adam\\
University of Exeter\\
\texttt{ap553@exeter.ac.uk}
\and
Haut, Terry\\
Lawrence Livermore National Laboratory\\
\and Wingate, Beth\\
University of Exeter\\
}
\title{Parareal Convergence for Oscillatory PDEs with Finite Time-scale Separation}

\begin{document}

\maketitle

\begin{abstract}
A variant of the Parareal method for highly oscillatory systems of PDEs was proposed by \cite{Haut_Wingate_14}. In that work they proved superlinear convergence of the method in the limit of infinite time scale separation. Their coarse solver features a coordinate transformation and a fast-wave averaging method inspired by analysis of multiple scales PDEs and is integrated using an HMM-type method. However, for many physical applications the timescale separation is finite, not infinite.  In this paper we prove convergence for finite timescale separaration by extending the error bound on the coarse propagator to this case. We show that convergence requires the solution of an optimization problem that involves the averaging window interval, the time step, and the parameters in the problem. We also propose a method for choosing the averaging window relative to the time step based as a function of the finite frequencies inherent in the problem.
\end{abstract}

\section{Introduction}

A variation of the Parareal method\cite{Lions_etal_01}\cite{Maday_Turinici_03} for highly oscillatory systems of equations was proposed in \cite{Haut_Wingate_14}. The method constructs the coarse solver based on a coordinate transformation and fast-wave averaging motivated by multiscale analysis of PDEs \cite{Bogoliubov_61, Klainerman_Majda_81, Schochet_94} and performs the integration using techniques from the Heterogeneous Multiscale Method \cite{Engquist_Tsai_05}. They proved that the method provides parallel speedups\cite{Haut_Wingate_14} in the limit of infinite time scale separation. Finite time scale separation is an important case to understand for physical applications because many physical phenomena, such as those occurring in numerical weather prediction, have finite frequencies inherent in the problem, e.g. Earth's rotation rate is finite. In this paper we extend the work of \cite{Haut_Wingate_14} by showing that rapid convergence of the method is also possible for finite timescale separation by proving error bounds on the coarse solver in the case of finite timescale separation.

Examples of applications of the Parareal algorithm being applied to parabolic PDEs include simulations of financial markets (i.e. the Black-Scholes equation for an American put \cite{Bal_Maday_02} and a nonlinear parabolic evolutionary equation via the finite element method \cite{He_10}. Hyperbolic systems solved with Parareal include simulation of molecular dynamics \cite{Baffico_etal_02}, fluid/structure interaction \cite{Farhat_Chandesris_03}, solution of the Navier-Stokes equations \cite{Fischer_etal_03}, and reservoir modelling \cite{Garrido_etal_03}. In all of these applications, the degree of oscillatory stiffness was not sufficient to impede convergence, but it is known (\textit{cf.} \cite{Ariel_etal_16}) to be an issue for the Parareal method.

There have been several modifications to the Parareal method which apply to highly oscillatory systems which assume that the system may be separated into fast and slow variables. In terms of ODEs, \cite{Legoll_etal_13} have proposed a multiscale method for singularly perturbed ODEs where the fast dynamics are dissipative. Ariel et al (2016)\cite{Ariel_etal_16} propose a method for highly oscillatory ODEs which is multiscale in nature but does not require explicit knowledge of the fast and slow variables. Gander and Hairer (2014) \cite{Gander_Hairer_14} suggest Parareal methods for Hamiltonian dynamics. Approaches using symplectic integrators with applications to molecular dynamics are presented in, for example, \cite{Audozze_etal_09} and \cite{Bal_Wu_08}. Finally, \cite{Haut_Wingate_14} proposed a method which is motivated by a asymptotic solutions for fast singular limits of nonlinear evolutionary PDEs. It is an extension of this method which we study here and which we refer to as Asymptotic Parallel-in-Time (APinT). It takes its name from the modified coarse solver which is inspired by methods used in the asymptotic analysis of PDEs.  

In this work we are primarily interested in oscillatory stiffness. Oscillatory stiffness places a restriction on the convergence of the Parareal method due to accuracy and stability limitations it places on the timestep size. We consider oscillatory stiffness to be a phenomenon arising from the presence of rapid oscillations which restricts the coarse timestep, with the degree of oscillatory stiffness being the degree to which the timestep is restricted. As discussed by Higham and Trefethen\cite{Higham_93}, stiffness is a transient phenomenon involving finite time intervals. This is important here as in this paper we are concerned with mitigating oscillatory stiffness over the interval of the coarse timestep.

We consider as a model equation a PDE of the form:

\begin{equation}
    \frac{\partial \mathbf{u}}{\partial t} + \frac{1}{\varepsilon}\mathcal{L}\mathbf{u} + \mathcal{N}(\mathbf{u},\mathbf{u}) = 0,
\label{eq:full_eqn}
\end{equation}
where $\mathbf{u}$ is the vector of unknowns, $\mathcal{L}$ is a skew-Hermitian linear operator with purely imaginary eigenvalues, and $\mathcal{N}(\cdot,\cdot)$ is a nonlinear operator which is assumed to be of quadratic type. We further assume that the solution, $\mathbf{u} \in L^{2}$ and that we may approximate \cref{eq:full_eqn} as a finite system of ODEs. The linear term then induces temporal oscillations on an $\mathcal{O}(\varepsilon)$ timescale, which can require the use of prohibitively small timesteps for standard numerical integrators if $\varepsilon$ is small. If accuracy is required, as is necessary for the Parareal method, implicit methods must also use a small time step.

The contribution of this work lies in an improved error estimate of the asymptotically motivated coarse solver which permits a mathematical description of the relationship between the time stepping error and the time-averaging of the the nonlinear operator. This understanding leads to an improved mathematical understanding of the convergence of the APinT method. With this improved notion of accuracy, we are then able to prove that the APinT algorithm converges for finite time scale separation. This is an advance on the previously shown case in the limit of $\varepsilon\to 0$.

The slow solution relies on an averaged version of \cref{eq:full_eqn}, with the average taken over an infinite window in the limit as $\varepsilon\to 0$. This is estimated numerically by a finite sum over a sufficiently large window. It has been shown\cite{E_03}\cite{Engquist_Tsai_05}\cite{Haut_Wingate_14} that the length required for this window may be reduced through the use of a smooth kernel of integration. In the small-$\varepsilon$ limit, we find that this method provides a convergent algorithm. We have also found that for finite $\varepsilon$ the averaging window may be chosen such that the slow solution is sufficiently accurate that the Parareal method remains convergent, as shown in~\cref{fig:iterations} and discussed in~\cref{sect:conv_any_eps}.

\begin{figure}[!htb]
    \centering
    \includegraphics[scale=0.55]{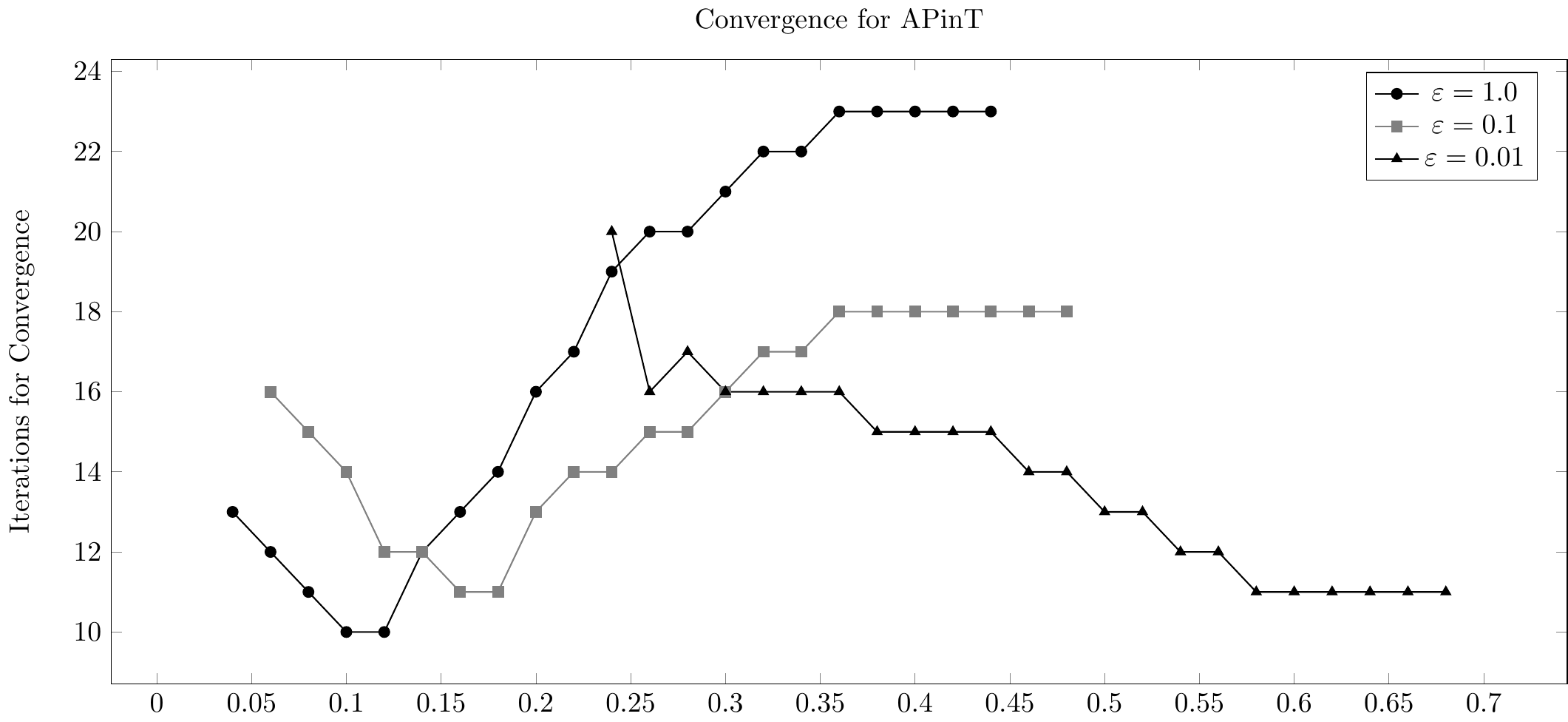}
    \caption{The number of iterations required for convergence of the APinT method for the 2-D rotating shallow water equations across three values of timescale separation, $\varepsilon$ (\textit{cf.} \cref{eq:full_eqn}). Note that towards the small-$\varepsilon$ limit, shown in red, the convergence improves with an increase in the size of the averaging window, as is consistent with the asymptotic theory. Of interest, however, there is a clear minimum is visible outside of this limit, especially for $\varepsilon = 1$, which marks a departure from the asymptotic theory where the limit is not taken. This makes clear both that the convergence of the method depends on the degree of scale separation, $\varepsilon$, and that the width of the averaging window, shown here proportional to the coarse timestep, may be chosen to control it. We shall rigorously explain this in~\cref{sect:bounds}. }
    \label{fig:iterations}
\end{figure}

In the next section we give an overview of the Parareal method to set the context of the work. In~\cref{sect:ass} we discuss the slow solution which is found by fast-wave averaging and which forms the coarse propagator of the APinT method. We will discuss the implications of the existence of near resonances on the slow solution. With that in mind, we proceed in~\cref{sect:bounds} to prove the error bounds and therefore the convergence of the APinT method. Finally, in~\cref{sect:rswe} we will show numerical experiments on the one-dimensional rotating shallow water equations and discuss some particularities of solving these equations with the APinT method.

\section{The Parareal Algorithm}
\label{sect:Parareal}
In this section we briefly review the Parareal method proposed by Lions et al.\cite{Lions_etal_01}, and further expanded upon by Maday and Turinici\cite{Maday_Turinici_03}. They proposed a generalisation of the concept of domain decomposition to the temporal domain, in which a `coarse' approximation to the solution is computed which is then refined, parallel in time, by the `fine' timestep. The solution has been shown to iteratively converge to the fine solution\cite{Gander_15}. In practice, this method requires that the coarse timestepping method permits large timesteps, that it be inexpensive to compute and sufficiently accurate that the method converges quickly. In general, the maximum timestep is $\mathcal{O}(\varepsilon)$, so in the case of $\varepsilon=\mathcal{O}(1)$, i.e. the less-stiff case, Parareal may be applied without any modifications. The insight of \cite{Haut_Wingate_14} was that a slow solution based on a coordinate transformation and a time average over the fast waves in the nonlinear operator provides a convergent and efficiently-computable coarse approximation. In fact, they showed that under suitable assumptions of smoothness, superlinear convergence is obtained as $\varepsilon\to 0$. 

Also related to this work is that of \cite{Maday_Turinici_05} who paid particular attention to the parallel implementation. They note that the coarse solver may employ a coarser timestep \cite{Lions_etal_01}, a coarser discretisation \cite{Fischer_etal_03}, and/or a simpler physical model \cite{Maday_Turinici_03}. The APinT coarse solver as presented here us a combination of the first and third of these. 

\begin{figure}[!htb]
\centering
\includegraphics[width=\textwidth]{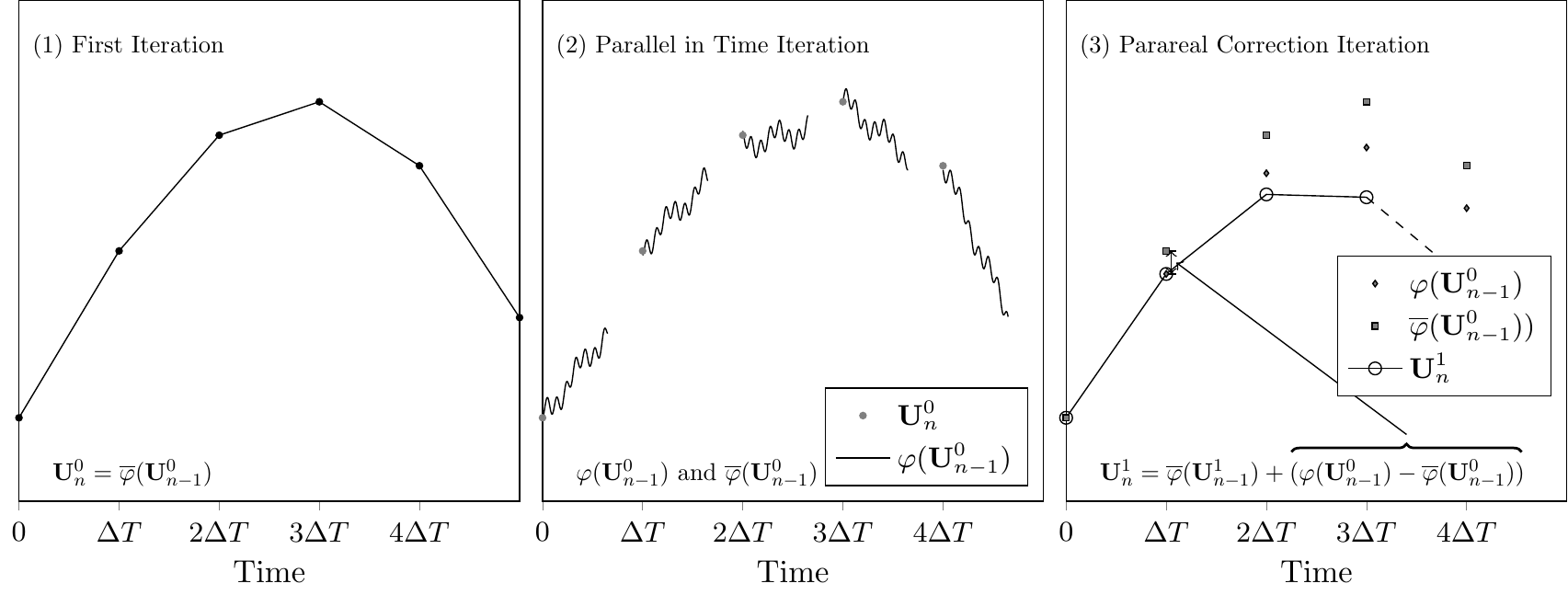}
    \caption{The Parareal algorithm. In (1), the coarse propagator is used to find an initial approximation to the solution. The solutions found by this iteration at coarse timesteps $n\Delta T$ are then taken as the initial conditions for parallel-in-time refinement by the fine proagator (2). Finally, the initial approximation is refined in a serial fashion by the Parareal correction iteration~\cref{eq:Parareal} shown in (3). The process in (2) and (3) is repeated, with the results of the correction iteration providing the new initial conditions for the next time-parallel iteration, until the desired level of convergence is obtained.}
\label{fig:Parareal}
\end{figure}

A sketch of the algorithm is in \cref{fig:Parareal}. We assume for the sake of simplicity that we are interested in solving~\cref{eq:full_eqn} on the interval $t\in[0,1]$. Let $\varphi_{t}(\mathbf{u}_{0})$ denote the evolution operator associated with~\cref{eq:full_eqn} such that $\mathbf{u}(t) = \varphi_{t}(\mathbf{u}_{0})$ solves the full equation. Similarly $\overline{\varphi}_{t}(\mathbf{u}_{0})$ solves the averaged equations.

We then divide the time domain into $N$ finite subintervals, $[n\Delta T, (n+1)\Delta T]$, where $n=0,\ldots N-1$. The Parareal algorithm begins with a coarse solve and then proceeds by computing approximations to the solution, $\mathbf{U}_{n}^{k}$, iteratively, following:

\begin{equation}
\mathbf{U}^{k}_{n} = \overline{\varphi}_{\Delta T}(\mathbf{U}^{k}_{n-1}) + (\varphi_{\Delta T}(\mathbf{U}^{k-1}_{n-1} - \overline{\varphi}_{\Delta T}(\mathbf{U}^{k-1}_{n-1})),\quad k=1,2,\ldots
    \label{eq:Parareal}
\end{equation}

Here, since the quantities $\mathbf{U}_{n-1}^{k-1}$ in the difference $(\varphi_{\Delta T}(\mathbf{U}^{k-1}_{n-1} - \overline{\varphi}_{\Delta T}(\mathbf{U}^{k-1}_{n-1}))$ are already computed at iteration $k$, the difference can be computed in parallel for all $n$. Since the computation of $\overline{\varphi}_{\Delta T}(\mathbf{U}_{n-1}^{k})$ is cheap, the overall computation is quick in a parallel sense if the iterates converge quickly.

\section{The Slow Solution}
\label{sect:ass}

Our interest in solving general PDEs which arise in physical modelling, in particular those of weather and climate, requires that we confront the problem of oscillatory stiffness over the interval of a coarse time step. In this section we shall describe the mathematical roots of the slow solution, provide a short description of its historical context, and review its discretisation as a coarse solver for the Parareal algorithm. As the convergence of the coarse solver for the APinT algorithm depends on the quality of the approximation, we will then investigate the numerical behaviour of this solver before providing an improved proof of the performance of the slow solver.

As an example of an application where the actual physics was thought to be asymptotic, but later shown not to be we look to the field of numerical weather prediction.  Historically, the physical notion of `slow' dynamics, called \textit{Quasi-Geostrophic} (QG) equations was a major advance in understanding weather. This insight was due to Charney\cite{Charney_48} who derived the `slow' equations. These reduced equations allowed the fast waves, which cause the oscillatory stiffness, to be filtered while still resolving the large-scale motions of the fluid. The work was later expanded upon, leading to what is generally recognised as the first successful numerical weather prediction \cite{Charney_49} \cite{Charney_Phillips_53}. Since those early days, the reduced equations have been rigorously shown to hold asymptotically in the limit of $\varepsilon \rightarrow 0$\cite{Embid_Majda_96}. In contrast to these results, modern weather prediction has found that the reduced equations are not accurate enough to be predictive and therefore they rely on numerical approximations of the full equations of motion\cite{Davies_etal_03}. As such, we are confronted with the problem that at least some of the oscillations matter even for the large scale flow and so we must find some way to resolve the fast waves in order to capture the full dynamics.

To address this problem for the Parareal method, \cite{Haut_Wingate_14} constructed a numerical approximation to the governing equations~\cref{eq:full_eqn} based on the consideration of fast singular limits\cite{Schochet_94}. One of the conclusions of this theory for when the linear operator, $\mathcal{L}$ is skew-Hermitian is that though the leading order dynamics is not slow itself, the slow dynamics evolve independently of the fast, while the fast dynamics are `swept' by the slow \cite{Schochet_94}. For realistic weather, we expect $\varepsilon \sim 0.01$ to $\varepsilon \sim 0.1$ \cite{Vallis_06}. Since $\varepsilon$ is finite for realistic cases, the timescale separation is also finite. Working in the limit of small $\varepsilon$ and applying the method of multiple scales, an averaged equation for equations of the form \cref{eq:full_eqn} was found by \cite{Embid_Majda_96}. 

For a slow timescale, $t$, and a fast timescale, $\tau$, they showed that averaged equations in the asymptotic limit of  $\varepsilon\to 0$  must satisfy:

\begin{equation}
    \begin{gathered}
        \frac{\partial\mathbf{\overline{u}}(\mathbf{x},t)}{\partial t} + \lim_{\tau \to \infty}\frac{1}{\tau}\int_{0}^{\tau}e^{s\mathcal{L}}\mathcal{N}(e^{-s\mathcal{L}}\mathbf{\overline{u}}(\mathbf{x},t),e^{-s\mathcal{L}}\mathbf{\overline{u}}(\mathbf{x},t))\,\mathrm{d}s  =0, \\
        \left.\mathbf{\overline{u}}(\mathbf{x},t)\right|_{t=0} = \mathbf{u}^{0}(\mathbf{x}),
    \end{gathered}
    \label{eq:nonlinear_averaged}
\end{equation}
where $\bar{\mathbf{u}}$ denotes the averaged $\mathbf{u}$ and where the integral is taken over the nonlinear operator, not the solution itself, and where there is a mapping by the exponential of the linear operator between the averaged and `full' solutions:
\begin{equation}
    \mathbf{u}^{0}(x,t,\tau) = e^{-\tau\mathcal{L}}\bar{\mathbf{u}}(\mathbf{x},t).
    \label{eq:leading_order}
\end{equation}
The above condition, studied in detail by \cite{Bogoliubov_61}, \cite{Klainerman_Majda_81}, and \cite{Schochet_94}, motivates our averaging, described below. We follow \cite{Haut_Wingate_14} and write the averaged equation in the following form:
\begin{equation}
    \frac{\partial \mathbf{\overline{u}}}{\partial t} + e^{t\mathcal{L}/\varepsilon}\overline{\mathcal{N}}(e^{-t\mathcal{L}/\varepsilon}\mathbf{\overline{u}}, e^{-t\mathcal{L}/\varepsilon}\mathbf{\overline{u}}) = 0,
    \label{eq:reduced_general_equation}
\end{equation}
from which the full solution may be obtained through the application of the matrix exponential as in~\cref{eq:leading_order}. The averaged equation, \cref{eq:reduced_general_equation}, has lost the factor of $1/\varepsilon$ from the full equation,~\cref{eq:full_eqn}, the main source of the oscillations, although another derivative will regain this term and so the oscillations have not been entirely eliminated.

In order to use this equation as a coarse solver for Parareal with finite timescale separation, \cite{Haut_Wingate_14} retreated from the asymptotic limit by taking the integral over the nonlinear operator in \cref{eq:nonlinear_averaged} over a finite time averaging window, rather than the infinite limit associated with $\epsilon \rightarrow 0$. This integral is approximated numerically by using a smooth kernel, $\rho(s)$, $0\le s \le 1$ which is chosen such that the length $T_{0}$ of the time window for the averaging is as small as possible, and approximate the averaged nonlinear operator~\cref{eq:nonlinear_averaged} as:
\begin{equation}
    \begin{aligned}
        \overline{\mathcal{N}}(\mathbf{\overline{u}}(t)) & \approx \frac{1}{T_{0}}\int_{0}^{T_{0}}\rho\left(\frac{s}{T_{0}}\right)e^{s\mathcal{L}}\mathcal{N}(e^{-s\mathcal{L}}\mathbf{\overline{u}}(t))\,\mathrm{d}s \\
        & \approx \frac{1}{\overline{M}}\sum\limits_{m=0}^{\overline{M}-1}\rho\left(\frac{s_{m}}{T_{0}}\right)e^{s_{m}\mathcal{L}}\mathcal{N}(e^{-s_{m}\mathcal{L}}\mathbf{\overline{u}}(t))
    \end{aligned}
    \label{eq:approx_avg}
\end{equation}

\begin{figure}[!htb]
    \centering
    \includegraphics[scale=0.7]{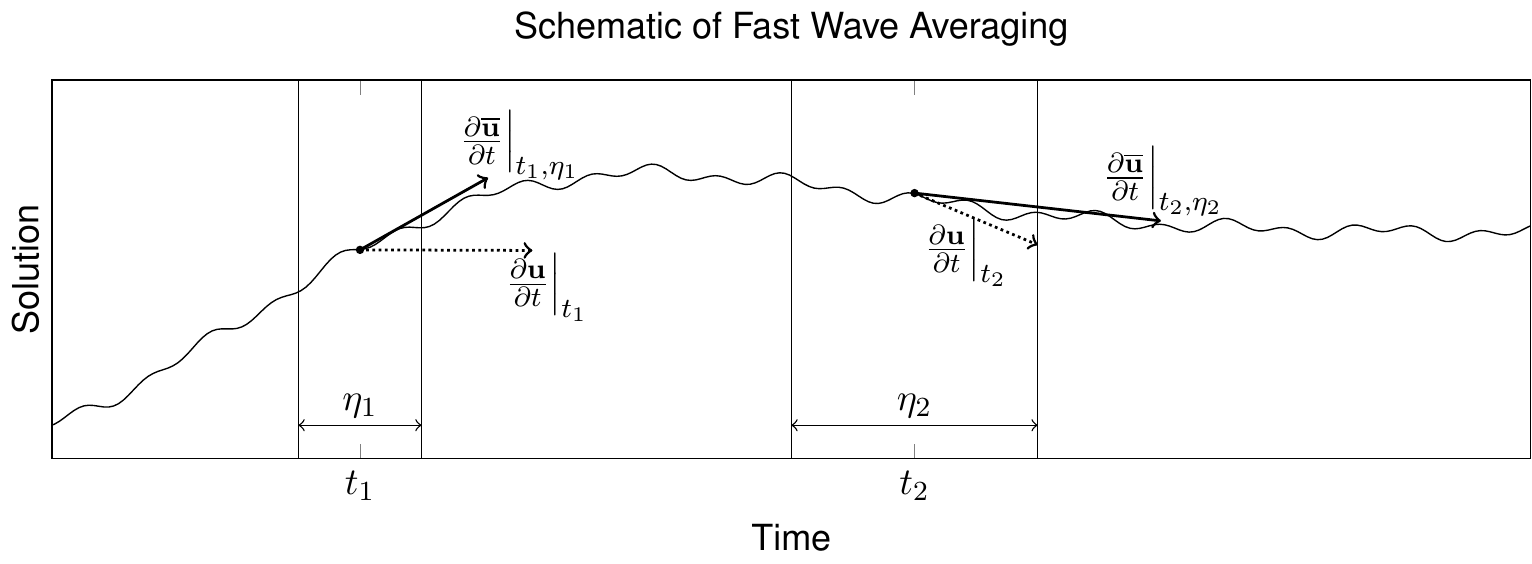}
    \caption{A schematic showing the averaged versus full derivatives for some toy solution exhibiting fast and slow behaviour, and for two different averaging window widths, $\eta_{1}$, and $\eta_{2}$. The solution here exhibits both rapid oscillations and a slower trend. Of importance here is the difference between the derivative considering the entire solution, i.e. $\frac{\partial\mathbf{u}}{\partial t}$ and that of the averaged solution with respect to the given averaging window, i.e. $\frac{\partial \mathbf{\overline{u}}}{\partial t}$. It is this averaged right-hand side which follows the slow solution without being affected by the fast oscillations and is therefore applicable to the coarse timestepping. }
    \label{fig:avg_window}
\end{figure}

The Heterogeneous Multiscale Method (HMM) \cite{Engquist_Tsai_05} is then applied to the slow equation by computing the averages numerically, as in \cref{eq:approx_avg}. In practice, the width of the averaging window may be freely chosen, as illustrated in~\cref{fig:avg_window}. The choice of this window has a significant effect on the convergence of the method, as illustrated in~\cref{fig:iterations}. As the computational cost of Parareal is proportional to the number of iterations required, an optimally-chosen window is necessary. In addition to a new error estimate our insight into the coarse error allows us to choose the optimal averaging window length as a function of the timescale separation.

\subsection{Triad Resonances}
\label{sect:triad}

In order to permit a long coarse timestep, the coarse solver proposed by \cite{Haut_Wingate_14} and described in~\cref{sect:ass} filters the nonlinear operator. The effect of this is a change in the content of the nonlinear triad interactions which is a function of both the degree of near resonance of the interaction and the length of the averaging window. As we discuss in \cref{sect:bounds}, the extent to which near-resonant sets are retained or rejected has an important impact on the convergence of the Parareal method.  Therefore, in this section we review nonlinear triad resonances for the model problem that we use in our tests in~\cref{sec:results}, noting that a similar approach applies to all systems of the form \cref{eq:full_eqn}.  

Systems governed by a quadratic nonlinearity with dispersive waves exhibit triad resonances\cite{Embid_Majda_96}\cite{Kramer_etal_02}\cite{Kadri_16}. Since we are motivated by geophysical modelling, we consider the general averaged equation for the rotating shallow water equations, which are commonly used as a test case for geophysical solvers. Following the notation of \cite{Embid_Majda_96}, we decompose the right-hand side of \cref{eq:nonlinear_averaged} in terms of its basis of eigenvectors and write:

\begin{equation}
\frac{\pd\overline{\vu}}{\pd t} = -\lim_{\tau\to\infty}\frac{1}{\tau}\int_{0}^{\tau} \sum\limits_{\kk\in\mathbb{Z}^{2}}\sum\limits_{\alpha = -1}^{1} \left[\sum\limits_{\kk=\kk_{1} + \kk_{2}}\sum\limits_{\alpha_{1},\alpha_{2}}\sigma_{\kk_{1}}^{\alpha_{1}}\sigma_{\kk_{2}}^{\alpha_{2}}C_{\kk,\kk_{1},\kk_{2}}^{\alpha,\alpha_{1}\alpha_{2}}e^{i(\kk\cdot\vx) - i\Omega_{\kk,\kk_{1},\kk_{2}}^{\alpha,\alpha_{1},\alpha_{2}}s/\varepsilon}\right]\mathbf{r}_{\kk}^{\alpha} \,\mathrm{d}s.
    \label{eq:triad_limit}
\end{equation}
where $\Omega_{\kk,\kk{1},\kk{2}}^{\alpha,\alpha_{1},\alpha_{2}} = \omega_{\kk_{1}}^{\alpha_{1}} + \omega_{\kk_{2}}^{\alpha_{2}} - \omega_{\kk}^{\alpha}$, $\alpha = -1,0,1$ refers to the different branches of the eigenvalues, $\kk$, $\kk_{1}$, and $\kk_{2}$ are the wavenumbers, $\omega_{\kk}^{\alpha}$ is the dispersion relation at a given $\alpha$ and wavenumber, $\sigma$ denotes the Fourier coefficient in this basis, $\mathbf{r}_{\kk}^{\alpha}$ is the right eigenvector of the linear operator, and $C_{\kk_{1},\kk_{2},\kk}^{\alpha_{1},\alpha_{2},\alpha}$ is an interaction coefficient \cite{Majda_Book}.

In the asymptotic case, the limit $\varepsilon\to 0$, by the orthogonality of the Fourier series, the only waves which remain after the wave averaging procedure (i.e. where $\tau\to\infty$) are the direct three-wave resonances (\textit{cf.} \cite{Embid_Majda_96}, \cite{Majda_Book}, \cite{Ward_Dewar_10}), i.e. the elements of the \textit{resonant set}, $\mathcal{S}_{\kk,\alpha}$ i.e.:

\begin{equation}
\mathcal{S}_{\kk,\alpha} = \{(\kk_{1},\kk_{2},\alpha_{1},\alpha_{2}):\kk = \kk_{1} + \kk_{2},\quad\omega_{\kk}^{\alpha} = \omega_{\kk_{1}}^{\alpha_{1}} + \omega_{\kk_{2}}^{\alpha_{2}}\}.
\end{equation}

The wave-averaged solution then follows:

\begin{equation}
\frac{\pd\sigma_{\kk}^{\alpha}}{\pd t} + \sum\limits_{\mathcal{S}_{\kk,\alpha}}\sigma_{\kk_{1}}^{\alpha_{1}}\sigma_{\kk_{2}}^{\alpha_{2}}C_{\kk_{1},\kk_{2},\kk}^{\alpha_{1},\alpha_{2},\alpha} = 0,
    \label{eq:wave_triad_form}
\end{equation}

It is this three-wave resonance condition from which the behaviour of the averaging kernel can be understood. In the limit as $\tau\to\infty$, only the direct resonances should remain. Because we have finite timescale separation,  this integral is approximated over a finite averaging window which must be large enough to filter the non-resonant triad.

As shown in \cite{Haut_Wingate_14}, using a finite time-averaging window and numerically integrating with respect to a smooth, finitely supported kernel permits this algorithm to result in a convergent Parareal algorithm in the limit of small $\varepsilon$.  

For finite $\varepsilon$ we take a finite average and so the solution set is larger than the direct resonant set and this has an important effect on the convergence of Parareal. To better explain the finite-$\varepsilon$ case, we define concentric shells of near-resonances, i.e. we rewrite the triad-based form~\cref{eq:triad_limit} as:

\begin{align}
    e^{s\mathcal{L}/\varepsilon}\mathcal{N}(e^{-s\mathcal{L}/\varepsilon}\overline{\mathbf{u}}(t),e^{-s\mathcal{L}/\varepsilon}\overline{\mathbf{u}}(t)) &= \sum_{\lambda_{n}}e^{i\lambda_{n}s}\mathcal{N}_{n}(\overline{\mathbf{u}}(t)),\\
    &= \sum_{\mathcal{S}_{\kk,\alpha}}\mathcal{N}_{n}(\overline{\mathbf{u}}(t)) + \sum_{\beta=1}^{\infty}\left(\sum_{\mathcal{S}_{\kk,\alpha}^{\epsilon_{\beta}}}e^{i\lambda_{n}s}\mathcal{N}_{n}(\overline{\mathbf{u}}(t))\right)
    \label{eq:near_res}
\end{align}
where $\mathcal{S}_{\kk,\alpha}^{\epsilon_{\beta}}$, $\beta = 1, 2, \ldots$ refers to a near-resonant set, i.e.:

\begin{equation}
    \mathcal{S}_{\kk,\alpha}^{\epsilon_{\beta}} = \left\{(\kk_{1},\kk_{2},\alpha_{1},\alpha_{2}):\kk = \kk_{1} + \kk_{2}, \quad\epsilon_{\beta-1} < \frac{1}{\varepsilon}|\omega_{\kk}^{\alpha} - \omega_{\kk_{1}}^{\alpha_{1}} + \omega_{\kk_{2}}^{\alpha_{2}}| \le \epsilon_{\beta}\right\},
    \label{eq:near_res_set}
\end{equation}
where $\epsilon_{0}=0$ by definition. The direct-resonant set results in a solution consisting of only the slow dynamics of the system -- which was shown in \cite{Embid_Majda_96} to be equivalent to the reduced equations for this system. Again, and as we will see in~\cref{sect:bounds}, the extent to which the near-resonant sets are retained and rejected by the averaging procedure is fundamental to the convergence of the APinT variation of the Parareal method.

\section{Error Bounds}
\label{sect:bounds}
Now that we have discussed the key elements of the algorithm, we are in a position to discuss and prove convergence for the case when $\varepsilon$ is finite. We shall first construct an improved error estimate for the coarse solution, and then use that result to prove the convergence of the Parareal method.
\subsection{A bound on the errors due to time-stepping and time-averaging in the coarse solver}
\label{sect:coarse}

In this section we employ the idea of near-resonant sets to extend the existing proof of APinT convergence\cite{Haut_Wingate_14} to the case of finite $\varepsilon$.  As has been demonstrated above (\textit{cf.}~\cref{fig:iterations}), the choice of the averaging window width, $\eta$, has a profound effect on the convergence of the method. While the choice of $\eta$ is well-understood for the limit of small $\varepsilon$\cite{Haut_Wingate_14}, we show here that $\eta$ may be similarly chosen to provide convergence for $\varepsilon$ up to $\mathcal{O}(1)$ for an appropriate coarse timestep. We first reduce~\cref{eq:full_eqn} to a standard form for ODEs. Following~\cref{sect:ass}, we write:

\begin{equation}
    \mathbf{v}_{t}(t) = e^{t\Lop/\varepsilon}N\left(e^{-t\Lop/\varepsilon}\mathbf{v}, e^{-t\Lop/\varepsilon}\mathbf{v}\right), \quad t \in [0,\Delta T],
    \label{eq:timescale}
\end{equation}
i.e. we are interested in the solution over a $\Delta T$ timescale. Let $\tau=t/(\varepsilon\Delta T)$, and so $\tilde{\mathbf{v}}\left(\tau\right)$, defined on the interval $\left[0,1/\varepsilon\right]$, 

\begin{equation}
\tilde{\mathbf{v}}\left(\tau\right)=\mathbf{v}\left(t\right).
\end{equation}

Then differentiation gives,

\begin{equation}
    \partial_{t}\mathbf{v}\left(t\right)=\partial_{t}\tilde{\mathbf{v}}\left(t/(\varepsilon\Delta T)\right)=\frac{1}{\varepsilon\Delta T}\partial_{\tau}\tilde{\mathbf{v}}\left(t/(\varepsilon\Delta T)\right)=\frac{1}{\varepsilon\Delta T}\partial_{\tau}\tilde{\mathbf{v}}\left(\tau\right).
\end{equation}

Upon this substitution into the coarse solver \cref{eq:reduced_general_equation} over the discrete time interval \cref{eq:timescale}, we arrive at the desired form which permits us to use the framework given in \cite{Sanders_Verhulst} where they have derived bounds for averaging methods. The aim of this is to modify and reapply their result for the error bound due to averaging, which holds on a general dynamical system of finite ODEs. This averaging error is one of the two major sources of error in the timestepping of the coarse solver. We then write the coarse solver in the form:

\begin{equation}
    \partial_{\tau}\tilde{\mathbf{v}}\left(\tau\right)=\varepsilon \Delta T e^{\tau \Delta T \Lop}N\left(e^{-\tau \Delta T \Lop}\tilde{\mathbf{v}}(\tau), e^{-\tau \Delta T \Lop}\tilde{\mathbf{v}}\left(\tau\right)\right).
    \label{eq:rescaled_full}
\end{equation}

While our interest is in solving PDEs describing physical systems, in practice we employ a Fourier spectral method, which has the effect of treating the PDE as a finite-dimensional system of ODEs. In this paper we show that the APinT method is convergent for finite systems of ODEs. This gives us access to the machinery of the numerical analysis of ODEs and averaging methods, following \cite{Sanders_Verhulst}. Let $\mathbf{x}$ solve the governing equations when they are written as a system of ODEs, i.e. in the form shown in~\cref{eq:rescaled_full}. For example, in the numerical experiments given in \cref{sec:results}, $\mathbf{x}$ is the Fourier solution. Then we may write:

\begin{equation}
\mathbf{x}_{t}=\varepsilon\mathbf{f}\left(\mathbf{x},t\right).
\label{eq:fine}
\end{equation}

Similarly, we consider the coarse solver \cref{eq:rescaled_full} written as a system of ODEs. Let $\mathbf{y}$ solve this averaged form of~\cref{eq:rescaled_full}, i.e.:

\begin{equation}
    \mathbf{y}_{t}=\varepsilon\overline{\mathbf{f}}\left(\mathbf{y},t\right),
\label{eq:coarse}
\end{equation}
where the averaging follows directly from the averaged equation,~\cref{eq:approx_avg} and is written:

\begin{equation}
    \overline{\mathbf{f}}\left(\mathbf{x},t\right)=\frac{1}{\eta}\int_{0}^{\eta}\rho\left(\frac{s}{\eta}\right)\mathbf{f}\left(\mathbf{y},t+s\right)ds,
    \label{eq:f_avg}
\end{equation}
where $\eta$ denotes the finite length of the averaging window.

\begin{lemma}
    Considering the initial value problems in $\mathbf{x}$ and $\mathbf{y}$ as stated above where $\mathbf{f}$ is $\mathbb{R}^{n} \times \mathbb{R}$ Lipschitz continuous with constant $\beta$ in $\mathbf{x}$ on $D \subset \mathbb{R}^{n}$ and $t$ on an $\mathcal{O}(1)$ timescale, i.e. for all $\mathbf{x}_{1},\mathbf{x}_{2}\in D$, $\beta$ is such that:
    
\begin{equation}
    \left\Vert \mathbf{f}\left(\mathbf{x}_{1},t\right)-\overline{\mathbf{f}}\left(\mathbf{\mathbf{x}}_{2},t\right)\right\Vert \leq\beta\left\Vert \mathbf{x}_{1}-\mathbf{x}_{2}\right\Vert.
\end{equation}

    Let:

\begin{equation}
M = \sup_{\mathbf{x}\in D}\sup_{0\leq t\leq L}\left\Vert \mathbf{f}\left(\mathbf{x},t\right)\right\Vert.
\end{equation}

    Then we can bound the difference between the exact solution $\mathbf{x}$ and the averaged solution $\mathbf{y}$ as:

\begin{equation}
\left\Vert \mathbf{x}-\mathbf{y}\right\Vert \leq M\left(1+\frac{1}{2}\beta \varepsilon\right)\varepsilon\Delta T\eta,
\end{equation}
    \label{lemma:avg_err}
\end{lemma}

    The above lemma follows from a modification of Lemma 4.2.8 in \cite{Sanders_Verhulst} in order to include the kernel of integration (\textit{cf.} appendix~\ref{sect:avg_err}). We have here bounded the error over an $\mathcal{O}(1)$ time interval instead of $\mathcal{O}(1/\varepsilon)$ so that the rate of convergence at different degrees of scale separation may be more easily compared, as in practice we are interested in simulations over fixed timescales. Taking the unmodified lemma provides a slightly different result as it gives the averaging error over a simulation time which scales with $\varepsilon$. Due to the numerical nature of the proof here, the appropriate timescale is over a coarse timestep.
    
\cref{lemma:avg_err} places a bound on the error committed by averaging over the fast waves, independent of the numerical methods used for spatial or temporal discretisation. Next, we consider the error arising from the numerical approximation of~\cref{eq:coarse}. In doing so, we will need to assume bounds on $\mathbf{f}$ in the region of phase space where $\mathbf{y}$ exists. Then assume that

\begin{equation}
    \|\partial_{\mathbf{y}}\mathbf{f}(\mathbf{y},t)\| \leq M_{1},\quad \mathbf{y}(t)\;\in\;D\subset\mathbb{R}^{n}.
    \label{eq:deriv_bounds}
\end{equation}

We assume that such a bound exists for higher spatial derivatives of $\mathbf{f}$, such that

\begin{equation}
    \max_{j}\left\|\frac{\partial^{j}\mathbf{f}}{\partial \mathbf{y}_{k}^{j}}\right\| \leq M, \quad 1 \leq k \leq n, \quad 0\leq j\leq p.
\end{equation}

\begin{lemma}
    Denote the numerical approximation to the averaged solution $\mathbf{y}(t)$ with timestep $\Delta T$ by a second-order timestepping method as $\mathbf{y}_{\Delta T}(t)$. Assume that $\mathbf{y}(t) = \varepsilon\mathbf{\overline{f}}(\mathbf{x},t)$ and that $\mathbf{\overline{f}} \in D$ as in \cref{eq:deriv_bounds}. Assume that integration is performed with respect to a smooth kernel, $\rho(\cdot)$, and let $\lambda_{n}$ denote the $n$-th near resonant triad (\textit{cf.} \cref{sect:triad}). Then the local time-stepping error of a second order time-stepping scheme applied to~\cref{eq:coarse} satisfies:

\begin{equation}
    \|\mathbf{y}(t) - \mathbf{y}_{\Delta T}(t)\| \le CM\varepsilon\Delta T^{3}\max_{x\in\mathbb{R}}\left(\lambda_{n}^{2}\frac{1}{\eta}\int_{0}^{\eta}\rho\left(\frac{s}{\eta}\right)e^{i\lambda_{n}s}ds\right),
\end{equation}
for some constant, $C \in \mathbb{R} < \infty$ and where $M$ is the bound over the nonlinear operator as given in \cref{lemma:avg_err}.
    \label{lemma:ts_err}
\end{lemma}

\begin{proof}
    The timestepping error of a $p$-th order scheme is bounded by \cite{Kincaid_Cheney_91}:

\begin{equation}
    \|\mathbf{y}(t) - \mathbf{y}_{\Delta T}(t)\| \le C_{t}\left(\Delta T\right)^{p+1}\max_{t}\left\Vert \frac{\mathrm{d}^{p+1}\mathbf{y}}{\mathrm{d}t^{p+1}}\left(t\right)\right\Vert _{2},
\end{equation}

   First, decompose $\mathbf{f}$ in terms of its basis of eigenvectors as discussed in \cref{sect:triad}. As with~\cref{eq:triad_limit}, we may write the solution as a sum of ODEs, each for a specific resonant nearness, $\lambda_{n}$. Then for the $j$-th component of $\mathbf{y}$, we write

\begin{equation}
\frac{\mathrm{d}y_{j}}{\mathrm{d}t} = \varepsilon\frac{1}{\eta}\int_{0}^{\eta}\rho\left(\frac{s}{\eta}\right)\sum_{n}\Delta T\eets\mathbf{N}_{n,j}(\yy)\,\mathrm{d}s,
\end{equation}
where the nearness of the resonances in any particular ODE is exposed through the eigenvalue sum, $\lambda_{n}$, in the exponent and where the subscript $,j$ denotes the $j$-th component and not a derivative, as it would with Einstein's notation. We then seek the third time derivative, which is found to be

\begin{multline}
\frac{\mathrm{d}^{3}y_{j}(t)}{\mathrm{d}t^{3}} = \varepsilon\pp \left(i^{2}\Delta T^{3}\sum_{n}\lambda_{n}^{2}\eets\mathbf{N}_{n,j} +\right. \\
\left. 2i\Delta T^{2}\sum_{n}\sum_{k}\lambda_{n}\eets\frac{\partial \mathbf{N}_{n,j}(\yy)}{\partial y_{k}}\frac{\mathrm{d}y_{k}(t)}{\mathrm{d}t}+\right.\\
\left.\Delta T\sum_{n}\sum_{k,l}\eets\frac{\partial^{2}\mathbf{N}_{n,j}(\yy)}{\partial y_{k}\partial y_{l}}\frac{\mathrm{d}y_{k}(t)}{\mathrm{d}t}\frac{\mathrm{d}y_{l}(t)}{\mathrm{d}t}+ \right. \\
\left.\Delta T\sum_{n}\sum_{k}\eets\frac{\partial\mathbf{N}_{n,j}(\yy)}{\partial y_{k}} \frac{\mathrm{d}^{2}y_{k}(t)}{\mathrm{d}t^{2}} \right)\,\mathrm{d}s.
\end{multline}

This is then the right-hand side which is integrated with respect to the smooth kernel. The magnitude of the near-resonant triad, $\lambda_{n}$, now presents itself as a multiplier on the complex exponential. It is then clear that it is this value, which is zero for direct resonances but becomes large in general, which is the source of numerical stiffness. For convenience, we introduce

\begin{equation}
P(\eta) = \pp\ees\,\mathrm{d}s; \quad e_{n}(t) = e^{i\Delta T\lambda_{n}t},
\end{equation}
then

\begin{multline}
\frac{\mathrm{d}^{3}y_{j}(t)}{\mathrm{d}t^{3}} = \varepsilon \Delta T^{3} P(\eta)\sum_{n}\left[-\lambda_{n}^{2}e_{n}\mathbf{N}_{n,j} +
 \left(2i\varepsilon\sum_{k}\lambda_{n}e_{n}\frac{\partial\mathbf{N}_{n,j}}{\partial y_{k}}\right)\left(\sum_{n'}e_{n'}\mathbf{N}_{n',j}\right) + \right. \\
\left. \varepsilon^{2}\sum_{k,l}e_{n}\frac{\partial^{2}\mathbf{N}_{n,j}}{\partial y_{k}\partial y_{l}}\left(\sum_{n'}e_{n'}\mathbf{N}_{n',j}\right)\left(\sum_{n''} e_{n''}\mathbf{N}_{n'',j}\right) + \right. \\
\left.\sum_{k'}e_{n}\frac{\partial \mathbf{N}_{n,j}}{\partial y_{k'}} \Biggl( \varepsilon\sum_{n'}\lambda_{n'}e_{n'}\mathbf{N}_{n',j} +
\varepsilon^{2}\sum_{n''}e_{n''}\mathbf{N}_{n'',j}\sum_{n'''}\sum_{l'}e_{n'''}\frac{\partial\mathbf{N}_{n''',j}}{\partial y_{l'}}\Biggr) \right],
\end{multline}

In bounding the timestepping error, we are interested in the norm of this quantity. Recalling that we are working with a finite-dimensional system of ODEs and applying the triangle and Cauchy-Schwarz inequalities we find that

\begin{multline}
\left\|\frac{\mathrm{d}^{3}y_{j}(t)}{\mathrm{d}t^{3}}\right\| \le \varepsilon\Delta T^{3}\|P(\eta)\|\sum_{n}|\mathbf{N}_{n,j}|\left(\|\lambda_{n}^{2}\| + \|2\lambda_{n}\varepsilon\|\left|\sum_{k}\frac{\partial \mathbf{N}_{n,j}}{\partial y_{k}}\right| +\right. \\
\left. \varepsilon^{2}|\mathbf{N}_{n,j}|\left|\sum_{k,l}\frac{\partial^{2}\mathbf{N}_{n,j}}{\partial y_{k}\partial y_{l}}\right|
\|\varepsilon\lambda_{n}\|\left|\sum_{k}\frac{\partial\mathbf{N}_{n,j}}{\partial y_{k}}\right| + \varepsilon^{2}\left|\frac{\partial\mathbf{N}_{n,j}}{\partial y_{k}}\right|^{2}  \right).
\end{multline}

Now, as $\mathbf{N}$ and all of its spatial derivatives up to and including $p=2$ are bounded by $M$ by~\cref{eq:deriv_bounds}, we write

\begin{align}
\left\|\frac{\mathrm{d}^{3}y(t)}{\mathrm{d}t^{3}}\right\| &\le \varepsilon\Delta T^{3}M\max_{x\in\mathbb{R}}P(\eta)\|\lambda_{n}^{2} + 3\lambda_{n}\varepsilon M + \varepsilon^{2}M^{2}\| \\
&\le \varepsilon\Delta T^{3}M\max_{x\in\mathbb{R}}P(\eta)\|\lambda_{n} + C_{f}\varepsilon M\|^{2},
\end{align}

where $C_{f}$ is a positive constant. We will now assume that $|\lambda_{n}|\neq 0$ as we are interested in the sup-norm of these values, which is nonzero when near-resonances are included. The directly resonant case has been treated by \cite{Haut_Wingate_14}. Then we must consider two possibilities. Firstly, if $|\lambda_{n}|\leq 1$, then we define some constant, $K_{1}$,

    \begin{equation}
        K_{1} = (1+C_{f}\varepsilon M)^{2}.
    \end{equation}
    
    If $|\lambda_{n}|>1$, the binomial theorem yields:
    
    \begin{align*}
        (|\lambda_{n}|+C_{f}\varepsilon M)^{p} &= \sum_{j=0}^{p}{{p}\choose{j}}(|\lambda_{n}|)^{p-j}(C_{f}\varepsilon M)^{j},\\
                                                     &\leq \sum_{j=0}^{p}{{p}\choose{j}}(|\lambda_{n}|)^{p}(C_{f}\varepsilon M)^{j},\\
                                                     &=|\lambda_{n}|^{p}\sum_{j=0}^{p}{{p}\choose{j}}(C_{f}\varepsilon M)^{j}, \\
                                                     &=|\lambda_{n}|^{p}K_{2}.
    \end{align*}

    And then we may write
    
    \begin{equation}
    (|\lambda_{n}| + \varepsilon\Delta TM)^{2} \leq \max(K_{1}, |\lambda_{n}|^{2}K_{2}).
    \end{equation}
    
    As for the Rotating Shallow Water Equations there must always be a value of $\lambda_{n}$ which is strictly greater than one, we shall assume that it is the second value which is the maximum. We now let $C = C_{t}K$. Finally, we bound the nonlinear term in the same fashion as \cref{lemma:avg_err}, where the fact that:

\begin{eqnarray*}
M & = & \sup_{\mathbf{x}\in D}\sup_{0\leq t\leq L}\left\Vert \mathbf{f}\left(\mathbf{x},t\right)\right\Vert \\
 & = & \sup_{\mathbf{x}\in D}\sup_{0\leq t\leq L}\left\Vert \sum_{n}\Delta T e^{i\Delta T \lambda_{n}t}\mathbf{N}_{n}\left(\mathbf{y}\right)\right\Vert \\
  & \leq & \sup_{\mathbf{x}\in D}\sup_{0\leq t\leq L}\left(\sum_{n}\left\Vert \mathbf{N}_{n}\left(\mathbf{y}\right)\right\Vert \right) < \infty,
\end{eqnarray*}
    completes the proof by providing an upper bound for the nonlinear operator as in \cref{lemma:avg_err}. This provides a bound for the error due to timestepping which does not depend directly on the solution, but rather on the general properties of the nonlinearity, in particular the triadic interactions.
    \end{proof}

With these two lemmas describing our primary sources of error, we will seek a bound on the error in the APinT algorithm and use this to show convergence.  From \cref{lemma:ts_err}, it follows that the timestepping error depends on:

\begin{equation}
    E\left(\varepsilon,\eta,\lambda_{n},\Delta T\right) \sim \varepsilon\Delta T\lambda_{n}^{2}\left(\frac{1}{\eta}\int_{0}^{\eta}\rho\left(\frac{s}{\eta}\right)e^{i\lambda_{n}\Delta Ts}ds\right).
\end{equation}

We define the following term which describes the filtering, independent of the gain due to the scale separation and the coarse timestep, and which is the key insight into understanding how to regularise an oscillatory problem over a finite time interval.

\begin{equation}
    \Lambda(\eta) = \max_{x\in\mathbb{R}}\lambda_{n}^{2}\int_{0}^{1}\rho(s)e^{i\lambda_{n}\eta\Delta Ts}\,\mathrm{d}s.
    \label{eq:big_lambda}
\end{equation}

$\Lambda(\eta)$ provides a measure of the extent to which the averaging integral mitigates the numerical stiffness. Recall that when the maximum $\lambda_{n}$ is large, as it is for highly oscillatory problems, it contributes to large gradients on the right-hand side requiring a small numerical timestep. In contrast, the integral component tends to zero as $\lambda_{n}$ gets large, and does so superlinearly because of the smooth kernel, $\rho(s)$ \cite{Haut_Wingate_14}. This term is then where we see precisely how the averaging procedure filters the fast oscillations, causing $\Lambda(\eta)$ to achieve a lower magnitude than $\lambda_{n}^{2}$ does on its own and therefore reducing the numerical stiffness.

In seeking a bound on the error in the timestepping, it is necessary to bound this term for some particular averaging kernel, $\rho(s)$. The choice of averaging kernel affects the error bounds through this function. $\Lambda(\eta)$ is bounded and tends rapidly to zero as $\eta \to \infty$ (\textit{q.v.}~\cref{sect:optimal_averaging}).

With this in mind, we now prove \cref{thm:coarse} which bounds the error committed by the coarse timestepping as compared to the fine. This will later allow us to prove error bounds on APinT subject to finite timescale separation.

\begin{theorem}
    Let $\Delta T$ denote the coarse timestep for a second order numerical method. We assume a finite scale separation on the order of $\varepsilon$. For an averaging window of length $\eta$, the total error in the coarse timestepping for the APinT algorithm is bounded by:

    \begin{equation}
        \|\mathbf{x}(t) - \mathbf{y}_{\Delta T}(t)\| \le M\varepsilon\Delta T\left((C_{0} + C_{1}\varepsilon)\eta + D_{1}(\Delta T)^{3}\Lambda(\eta)\right),
    \end{equation}
    where $M$ is the sup-norm over the nonlinear operator as in lemmas~\ref{lemma:avg_err} and \ref{lemma:ts_err} and $C_{0}$, $C_{1}$, and $D_{1}$ are finite constants.

    \label{thm:coarse}
\end{theorem}

\begin{proof}

    By the triangle inequality, we may write:

\begin{eqnarray*}
\left\Vert \mathbf{x}\left(t\right)-\mathbf{y}_{\Delta T}\left(t\right)\right\Vert  & = & \left\Vert \mathbf{x}\left(t\right)-\mathbf{y}\left(t\right)+\mathbf{y}\left(t\right)-\mathbf{y}_{\Delta T}\left(t\right)\right\Vert, \\
 & \leq & \left\Vert \mathbf{x}\left(t\right)-\mathbf{y}\left(t\right)\right\Vert +\left\Vert \mathbf{y}\left(t\right)-\mathbf{y}_{\Delta T}\left(t\right)\right\Vert.
\end{eqnarray*}

\cref{lemma:avg_err} is used to bound the first term, i.e.:
    
    \begin{equation}
        \|\mathbf{x}(t) - \mathbf{y}(t)\| \le M(C_{0} + C_{1}\varepsilon)\varepsilon\Delta T \eta.
        \label{eq:first_bound}
    \end{equation}
    
    Applying~\cref{lemma:ts_err} and~\cref{eq:big_lambda} to the second term yields:

    \begin{align}
        \|\mathbf{y}(t) - \mathbf{y}_{\Delta T}(t)\| &\le MCC_{1}(\Delta T)^{3}\varepsilon\Lambda(\eta),\\
        &\le MD_{1}(\Delta T)^{3}\varepsilon\Lambda(\eta),
        \label{eq:second_bound}
    \end{align}
    where $\Lambda(\eta)$ is bounded independently of $\lambda_{n}$ for any averaging window length, $\eta$. Combining the bounds in equations~\cref{eq:first_bound} and \cref{eq:second_bound} gives the theorem as desired.

\end{proof}

\subsection{Proof of Parareal Convergence}
\label{sect:conv}

We may now derive error bounds for the Parareal iteration on finite systems of ODEs, given in~\cref{eq:Parareal}. Using our improved error bound for the coarse solver which holds for finite $\varepsilon$, we modify the proof given in \cite{Haut_Wingate_14}, which held only as $\varepsilon\to 0$. For consistency we define several operators following \cite{Haut_Wingate_14}. Let $\tilde{\varphi}_{\Delta T}(\cdot)$ be the evolution operator associated with numerically solving the slow equation using an $\mathcal{O}(p)$ method, such that $\tilde{\varphi}_{\Delta T}(\cdot)$ is a numerical approximation of $\overline{\varphi}_{\Delta T}(\cdot)$. Furthermore, let $\varphi_{\Delta T}(\cdot)$ denote the evolution operator for the fine solution. We then define:

\begin{equation}
    \mathcal{E}_{\varphi,\overline{\varphi}}(\cdot) = \varphi_{\Delta T}(\cdot) - \overline{\varphi}_{\Delta T}(\cdot);\quad \mathcal{E}_{\overline{\varphi},\tilde{\varphi}}(\cdot) = \overline{\varphi}_{\Delta T}(\cdot) - \tilde{\varphi}_{\Delta T}(\cdot),
\end{equation}

Then, as in \cite{Bal_05}, \cite{Haut_Wingate_14}, and~\cref{sect:coarse} we make the following assumptions:

\begin{enumerate}
    \item
        The operators $\varphi(\cdot)$ and $\overline{\varphi}(\cdot)$ are uniformly bounded for $0\le t \le 1$:
        \begin{equation}
            \|\varphi_{t}(\mathbf{u}_{0})\|\le C\|\mathbf{u}_{0}\|\, , \quad \|\overline{\varphi}_{t}(\mathbf{u}_{0})\|\le C\|\mathbf{u}_{0}\|
            \label{eq:uniform_bound}
        \end{equation}
    \item
        The averaging method is accurate in the sense that:
        \begin{equation}
            \|\varphi_{t}(\mathbf{u}_{0}) - \overline{\varphi}_{t}(\mathbf{u}_{0})\| \le \varepsilon\Delta T\eta M(C_{1} + C_{2}\varepsilon)\|\mathbf{u}_{0}\|
            \label{eq:asymp_acc}
        \end{equation}
    \item
        The averaged evolution operator satisfies:
        \begin{equation}
            \|\overline{\varphi}_{\Delta T}(\mathbf{u}_{1}) - \overline{\varphi}_{\Delta T}(\mathbf{u}_{2})\| \le (1 + C\Delta T)\|\mathbf{u}_{1} - \mathbf{u}_{2}\|,
        \end{equation}
        and the numerical approximation to the evolution equation satisfies:
        \begin{equation}
            \|\tilde{\varphi}_{\Delta T}(\mathbf{u}_{1}) - \tilde{\varphi}_{\Delta T}(\mathbf{u}_{2})\| \le (1 + C\Delta T)\|\mathbf{u}_{1} - \mathbf{u}_{2}\|,
            \label{eq:num_evol_bound}
        \end{equation}
    \item
        Following \cref{lemma:avg_err} and \cref{lemma:ts_err} and~\cref{eq:uniform_bound}, the error operators satisfy:

        \begin{equation}
            \|\mathcal{E}_{\varphi,\overline{\varphi}}(\mathbf{u}_{1}) - \mathcal{E}_{\varphi,\overline{\varphi}}(\mathbf{u}_{2})\| \le \varepsilon\Delta T \eta M (C_{1} + C_{2}\varepsilon) \|\mathbf{u}_{1} - \mathbf{u}_{2}\|,
            \label{eq:err_op_asymp}
        \end{equation}
        and:
        \begin{equation}
            \|\mathcal{E}_{\overline{\varphi},\tilde{\varphi}}(\mathbf{u}_{1}) - \mathcal{E}_{\overline{\varphi},\tilde{\varphi}}(\mathbf{u}_{2})\| \le \Delta T^{3}\varepsilon\Lambda(\eta) M C \|\mathbf{u}_{1} - \mathbf{u}_{2}\|, \quad p \ge 1.
            \label{eq:err_op_ts}
        \end{equation}

\end{enumerate}

We have now quantified the major sources of error in the coarse timestepping which will affect the convergence of Parareal. The following proof of the convergence follows directly from these bounds.

\begin{theorem}
    Subject to the above assumptions, the error, $\mathbf{u}(T_{n})-\mathbf{U}_{n}^{k}$, after the $k$-th Parareal iteration is bounded by:

    \begin{equation}
        \|\mathbf{u}(T_{n}) - \mathbf{U}_{n}^{k}\| \le MC_{g}\left(C_{1}\Delta T^{3}\varepsilon\Lambda(\eta) + (C_{2} + C_{3}\varepsilon)\varepsilon\eta\right)^{k+1}\|\mathbf{u}_{0}\|.
    \end{equation}
    \label{thm:apint_conv}
\end{theorem}

\begin{proof}
The proof is by induction on $k$. When $k=0$:

    \begin{align*}
        \|\mathbf{u}(T_{n})-\mathbf{U}_{n}^{n}\| &= \|\varphi_{\Delta T}(\mathbf{u}_{0}) - \tilde{\varphi}_{\Delta T}(\mathbf{u}_{0})\| \\
        & \le \|\varphi_{\Delta T}(\mathbf{u}_{0}) - \overline{\varphi}_{\Delta T}(\mathbf{u}_{0})\| + \|\overline{\varphi}_{\Delta T}(\mathbf{u}_{0}) - \tilde{\varphi}_{\Delta T}(\mathbf{u}_{0})\| \\
        & \le M((C_{1} + C_{2}\varepsilon)\varepsilon\Delta T \eta + C_{3}\Delta T^{2})\|\mathbf{u}_{0}\|,
    \end{align*}

    where we have used~\cref{eq:asymp_acc}, which bounds the error induced by the averaging procedure, to bound the first term and \cref{lemma:ts_err}, which governs the timestepping error, for the second. Now assume that:

    \begin{equation}
        \|\mathbf{u}(T_{n})-\mathbf{U}_{n}^{k-1}\| \le (\Delta T + \varepsilon)\left(C_{1}\Delta T^{3}\varepsilon\Lambda(\eta) + (C_{2}+C_{3}\varepsilon)\varepsilon\eta\right) \|\mathbf{u}_{0}\|.
    \end{equation}

    We may then write the Parareal iteration,~\cref{eq:Parareal} in the following form, using~\cref{eq:err_op_asymp} and \cref{eq:err_op_ts}:

    \begin{multline}
        \mathbf{u}(T_{n}) - \mathbf{U}_{n}^{k} = (\tilde{\varphi}_{\Delta T}(\mathbf{u}(T_{n-1})) - \tilde{\varphi}_{\Delta T}(\mathbf{U}_{n-1}^{k})) + \\
        (\mathcal{E}_{\varphi,\overline{\varphi}}(\mathbf{u}(T_{n-1})) - \mathcal{E}_{\varphi,\overline{\varphi}}(\mathbf{U}_{n-1}^{k-1})) + (\mathcal{E}_{\overline{\varphi},\tilde{\varphi}}(\mathbf{u}(T_{n-1})) - \mathcal{E}_{\overline{\varphi},\tilde{\varphi}}(\mathbf{U}_{n-1}^{k-1}))
    \end{multline}

    By directly substituting equations~\cref{eq:num_evol_bound}, \cref{eq:err_op_asymp}, and \cref{eq:err_op_ts}, we have:

    \begin{align*}
        \|\mathbf{u}(T_{n}) - \mathbf{U}_{n}^{k}\| &\le (1+C\Delta T)\|\mathbf{u}(T_{n-1})-\mathbf{U}_{n-1}^{k}\| + \\
        &\qquad M\left(C_{1}\Delta T^{3}\varepsilon\Lambda(\eta) + (C_{2} + C_{3}\varepsilon)\varepsilon\Delta T\eta\right)\|\mathbf{u}(T_{n-1})-\mathbf{U}_{n-1}^{k-1}\| \\
        & \le (1+C\Delta T)\|\mathbf{u}(T_{n-1})-\mathbf{U}_{n-1}^{k}\| + \\
        &\qquad M\Delta T\left(C_{1}\Delta T^{2}\varepsilon\Lambda(\eta) + (C_{2} + C_{3}\varepsilon)\varepsilon\eta\right)^{k+1}\|\mathbf{u}_{0}\|.
    \end{align*}

    Finally, application of the discrete Gronwall inequality gives:

    \begin{align*}
        \|\mathbf{u}(T_{n}) - \mathbf{U}_{n}^{k}\| &\le \left(e^{C(T_{n}-T_{0})}-1\right)M \left(C_{1}\Delta T^{2}\varepsilon\Lambda(\eta) + (C_{2} + C_{3}\varepsilon)\varepsilon\eta\right)^{k+1}\|\mathbf{u}_{0}\|\\
         &\le MC_{g}\left(C_{1}\Delta T^{2}\varepsilon\Lambda(\eta) + (C_{2} + C_{3}\varepsilon)\varepsilon\eta\right)^{k+1}\|\mathbf{u}_{0}\|.
    \end{align*}
\end{proof}

\cref{thm:apint_conv} is one of the key contributions of this work. Using the understanding of near-resonance and the result of \cref{thm:coarse}, it generalises the proof given by 
\cite{Haut_Wingate_14} of convergence for
the asymptotic limit as $\varepsilon \to 0$ to finite $\varepsilon$. This is a significant improvement as for many physical
applications such as weather and climate modelling $\varepsilon$
remains finite. As the averaging window length, $\eta$,
may be freely chosen we may select an optimal $\eta$
for a wide range of $\varepsilon$
subject to the other constants and choice of $\Delta T$
such that the method is convergent. We discuss this in the next section.


\subsection{Convergence for any $\varepsilon$}
\label{sect:conv_any_eps}

Given \cref{thm:apint_conv} we are in finally in a position to discuss convergence for any timescale separation. For the APinT algorithm to converge, we require that:

\begin{equation}
    C_{1}\Delta T^{3}\varepsilon\Lambda(\eta) + C_{2}\varepsilon\eta + C_{3}\varepsilon^{2}\eta < 1,
\end{equation}

We are then left with the problem of choosing an appropriate averaging window length, $\eta$, depending on the degree of scale separation, $\epsilon$, and the filtered contribution of the triads, $\Lambda(\eta)$. In the interest of demonstrating that one exists, we assume the scaling (for example):

\begin{equation}
    \eta = \frac{\Delta T}{\varepsilon^{s}},\quad 0<s<1.
    \label{eq:eta_scaling}
\end{equation}

We then have:

\begin{equation}
    C_{1}\Delta T^{3}\varepsilon\Lambda\left(\frac{\Delta T}{\varepsilon^{s}}\right) + C_{2}\varepsilon^{1-s}\Delta T + C_{3}\varepsilon^{2-s}\Delta T < 1,
\end{equation}
as $\varepsilon \to 0$, our error also decreases for any value of the power $s$. $\Lambda\left(\frac{\Delta T}{\varepsilon^{s}}\right)$ is bounded, so as $\varepsilon \to 1$, all terms remain bounded and we may choose our coarse timestep accordingly to ensure convergence. This means that the method proposed here may be applied across the full range of $\varepsilon \in (0, 1]$ with only a change of averaging window length, which allows convergence for physical problems where the time scale separation may change throughout the computation. This is in contrast to the proof in the limit \cite{Haut_Wingate_14} which proved convergence only for $\varepsilon\to 0$.

\section{The One-Dimensional Rotating Shallow Water Equations}
\label{sect:rswe}

We now consider an example, using the one-dimensional rotating shallow water equation as a test-case, as did \cite{Haut_Wingate_14}. Let the unknown vector be:

\begin{equation}
    \mathbf{u}(t,x) = \left(v_{1}(t,x), v_{2}(t,x), h(t,x)\right)^{T}.
\end{equation}

We then write the linear and nonlinear operators in the full model~\cref{eq:full_eqn} as:

\begin{equation}
    \mathcal{L} = \left(\begin{array}{ccc}0 & -1 & F^{-1/2}\partial_{x} \\ 1 & 0 & 0 \\ F^{-1/2}\partial_{x} & 0 & 0\end{array}\right); \quad \mathcal{N}(\mathbf{u},\mathbf{u}) = \left(\begin{array}{c}v_{1}(v_{1})_{x}\\v_{1}(v_{2})_{x}\\(hv_{1})_{x}\end{array}\right).
\end{equation}
for some constant, $F\in\mathbb{R}$. The corresponding eigenvalues are:

\begin{equation}
    \omega_{k}^{\alpha} = \alpha\sqrt{1+F^{-1}k^{2}},\quad\alpha=-1,0,+1.
    \label{eq:1d_eigenval}
\end{equation}

In general, as $\varepsilon \to 0$ we expect that $\Lambda(\eta) \to 0$ as well due to cancellation of oscillations in the integral\cite{Embid_Majda_96}. As discussed in \cref{sect:bounds}, oscillatory stiffness arises due to the magnitude of the gain term outside of the integral, which is large for highly oscillatory systems. The integral itself, however, is bounded from above by one, and achieves this value only for directly resonant triads (\textit{cf.} \cref{sect:triad}), where the gain is zero. As the distance of resonance (i.e. the magnitude of $|\omega_{k}^{\alpha}-\omega_{k_{1}}^{\alpha_{1}} - \omega_{k_{2}}^{\alpha_{2}}|$) increases, the integral tends to zero as well (and does so faster with larger $\eta$).

The choice is then for a given degree of scale separation, $\varepsilon$, to choose an $\eta$ which mitigates the stiffness sufficiently to allow the necessary coarse timestep, while retaining as much fidelity to the full equations as possible (\textit{cf.} \cref{lemma:avg_err}, where the averaging error is proportional to $\eta$). We shall deal with the practical implications of the form of $\Lambda(\eta)$ in \cref{sect:optimal_averaging}.

\subsection{The Optimally-Averaged Slow and Fast Solutions}

In order to illustrate the slow averaged solution over which the timestepping is performed and its relation to the full solution,~\cref{fig:timelapse_fast} compares the slow, full and true solutions for the stiff case where $\varepsilon=0.01$. The spatio-temporal oscillations are very rapid in the stiff case, which is the source of the timestep limitation by the CFL condition. However, the slow solution over which the timestepping is performed lacks these rapid oscillations and so permits the large timestep.

Spatio-temporal oscillations from an initially stationary Gaussian height field are shown on a domain which is spatially periodic, i.e. the top and bottom boundaries of the plots wrap around. The decay of the height field into waves travelling in opposite directions is visible in~\cref{fig:timelapse_fast}. The optimal averaging window (\textit{q.v.}~\cref{sect:optimal_averaging}) was applied, and convergence to single precision was obtained in six iterations.

\begin{figure}[H]
    \centering
    \includegraphics[scale=0.55]{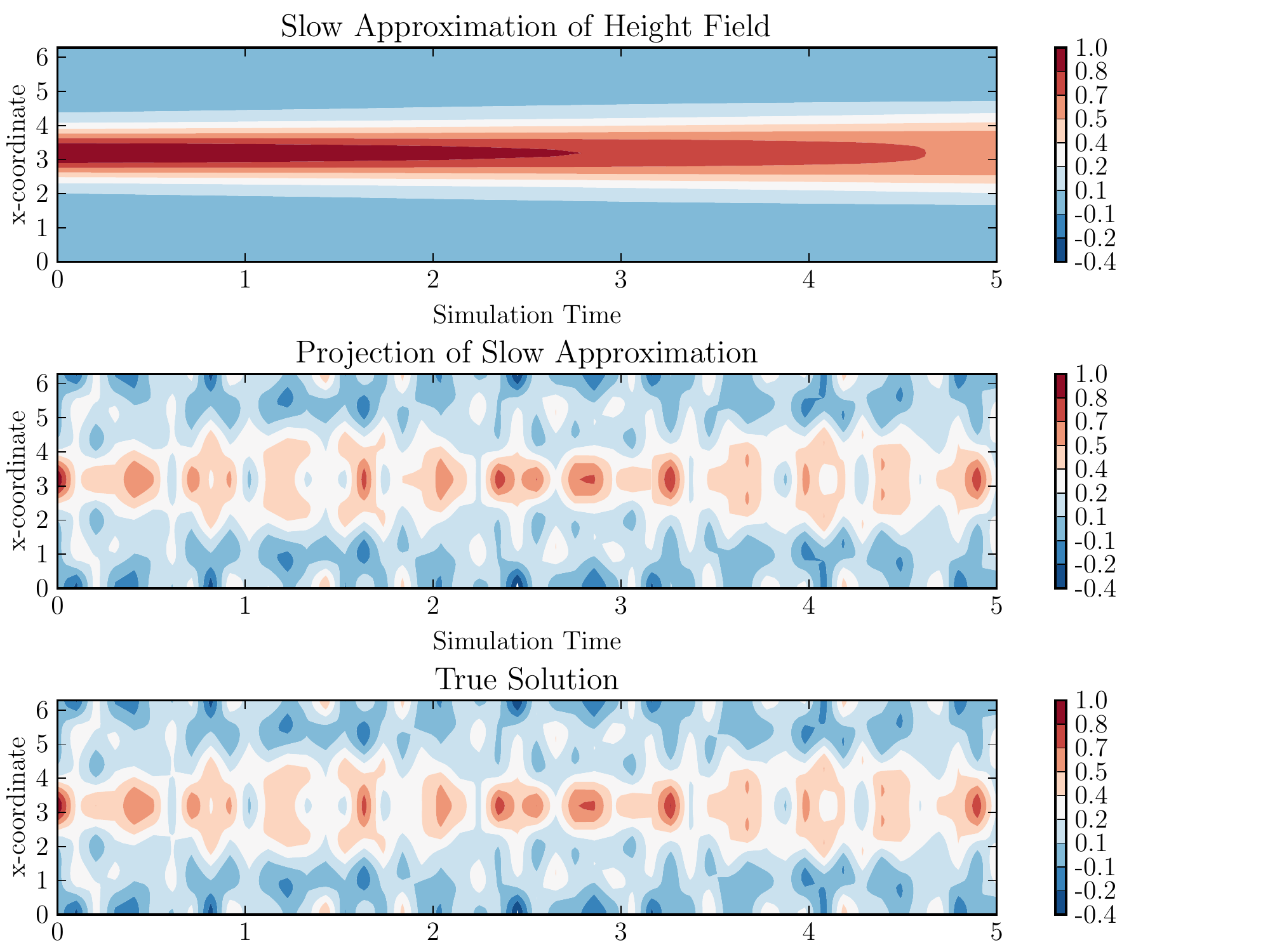}
    \caption{A comparison of the solution derived from the averaging method with the `true' solution. All three plots show spatio-temporal oscillations in the height field of the 1D RSWE with the time coordinate on the $x$-axis and the spatial coordinate on the $y$-axis. The top plot shows the slow approximation of the height field (the third component of $\mathbf{\overline{u}}$). The timestepping is performed over this slower quantitywith decreased oscillatory stiffness and therefore an increase in timestep. The middle plot is the projection of this quantity back into normal space by the matrix exponential, $e^{\tau\mathcal{L}}$. This is the coarse solution which is used after the first coarse solve. The quality of this when compared to the solution computed with the fine solver shown in the last plot (to which the APinT algorithm converges) is what allows rapid convergence of Parareal. In this example, $\varepsilon=0.01$ and the averaging window, $\eta = 1.0$, which is optimal for this problem..}
    \label{fig:timelapse_fast}
\end{figure}

\subsection{Numerical Results on the 1-D RSWE}
\label{sec:results}

In this section we present numerical results for the one-dimensional rotating shallow water equations which build on those presented in \cite{Haut_Wingate_14}. \Cref{fig:phi} shows the norm of the coarse error, i.e. $\|\mathbf{x}(t) - \mathbf{y}(t)\|_{2}$ computed relative to the fine timestep versus the width of the averaging window, $T_{0}$, which denotes the numerical choice of $\eta$, where $0\le t\le 1$, $\delta t = 2\mathrm{e}{-4}$, and the spatial resolution is $N_{x} = 64$. This spatial resolution and fine timestepping regime were found to be within the asymptotic range of the timestepping. A second-order Strang splitting method was used for both the coarse and fine solves. The initial flow was stationary with a Gaussian height field.

For the smallest $\varepsilon$ the asymptotic behaviour is well approximated, as the fidelity of the coarse timestepping increases as the averaging window increases. This is consistent with the behaviour described in~\cref{sect:triad}, where the theory predicts that $\eta\to\infty$ as $\varepsilon\to 0$. However, for larger $\varepsilon$ such as the two cases shown, there is a clear optimal size for the averaging window to take, i.e. the minimum in the red and green curves in~\cref{fig:phi}. The location but not the magnitude of this point is predicted by~\cref{eq:phi_min}. 

\Cref{thm:coarse} states that we should expect that outside of the small-$\varepsilon$ limit the iterative error should decrease with $k$ as in \cref{thm:apint_conv} and exhibit a minimum where the sum of the timestepping and averaging errors is smallest. In the case of $\varepsilon = 0.01$, which is near the limit as $\varepsilon\to 0$, we expect that for a large enough averaging window we will have optimal convergence, with no improvement in solution quality for a larger averaging window. The numerical results are then consistent with the theory developed in this paper.

\begin{figure}[H]
    \centering
    \includegraphics[scale=0.55]{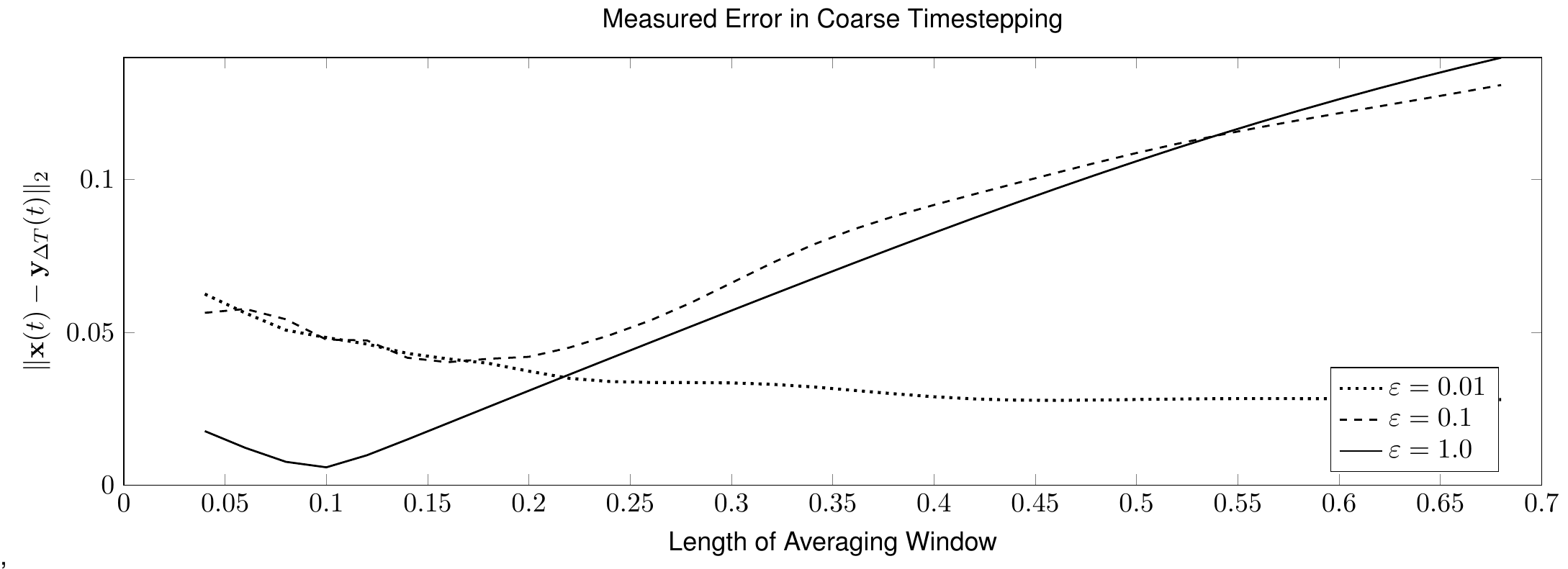}
    \caption{Computed coarse error for $\Delta T = 0.1$. This is a numerical estimate of the error corresponding to that in \cref{thm:coarse}, i.e. $\|\mathbf{x}(t) - \mathbf{y}_{\Delta T}(t)\|$, computed by brute-force comparison of the averaged coarse solution to a finely computed reference solution. Note the clear existence of an optimal averaging window for the case where $\varepsilon=1.0$, and the tendency towards the asymptotic theory, i.e. the error becoming inversely proportional to the averaging window length, $T0$, as $\varepsilon\to 0$.}
    \label{fig:phi}
\end{figure}

In practice, we seek a choice for $T_{0}$ for which the solution is non-stiff on an $\mathcal{O}(\Delta T)$ interval, and therefore as $\Delta T$ increases, so must the averaging window. Similarly, the oscillatory stiffness is proportional to $1/\varepsilon$, and so as $\varepsilon\to 0$ it is necessary to choose a longer averaging window, and therefore to apply stronger smoothing to the solution.

Comparing~\cref{fig:phi} to~\cref{fig:apint_conv}, which shows the iterative error in the APinT method after three iterations for the same parameters, the direct computation of the coarse timestepping error provides good qualitative agreement with the optimal choice of $\eta$ for the different values of $\varepsilon$. This is in direct agreement with the prediction of \cref{thm:apint_conv}.

As $T_{0}$ is taken smaller, instability is observed for all $\epsilon$. This corresponds to the explicit CFL limit being violated, as reducing the length of the averaging window increases the maximum wave speed in the solution. For large $T_{0}$, the iterative error roughly stabilises for finite $\varepsilon$. In the limit as $\eta$ is taken very large, the coarse timestepping corresponds to an incorrect equation (e.g. as in the QG equations for our example) being solved in a numerically stable fashion. The difference in the coarse and fine equations is sufficient to inhibit convergence, but does not violate the timestepping limit.

\begin{figure}[H]
    \centering
    \includegraphics[scale=0.55]{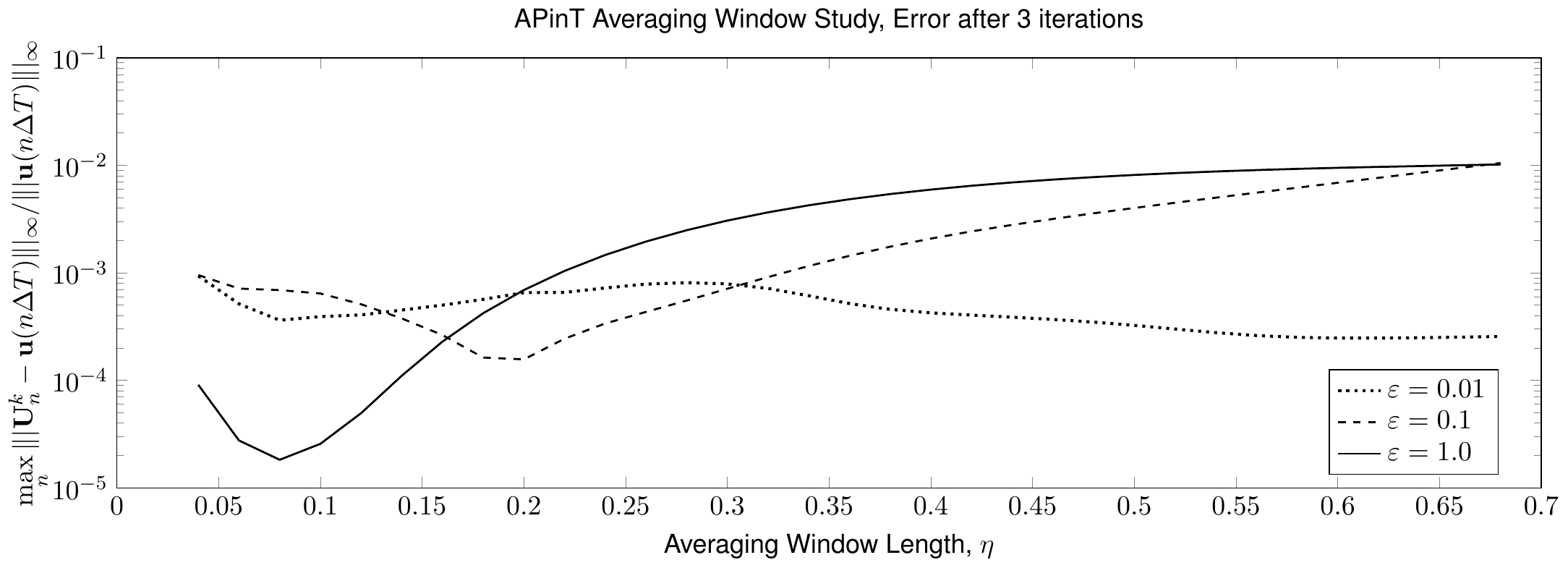}
    \caption{Iterative error in APinT with 1-D RSWE after three iterations for $\Delta T = 0.1$. Whereas~\cref{fig:phi} showed the measured total (i.e. timestepping plus averaging) error in the coarse timestepping, this figure shows the iterative error for a full Parareal solve of the RSWE after three iterations for the same computational conditions. Note that the behaviour with respect to variation of the averaging window and particularly the location of the optimal window length is well predicted by the brute force computation of the coarse timestepping error.}
    \label{fig:apint_conv}
\end{figure}

\subsection{Optimal Averaging for the 1-D RSWE}
\label{sect:optimal_averaging}

It was shown in~\cref{sect:conv_any_eps} that it is possible to choose the averaging window in such a way as to ensure convergence. Beyond doing this, we may choose the window optimally to obtain the fastest possible convergence (\textit{cf.}~\cref{fig:apint_conv}).

The reason we are able to describe the qualitative behaviour this way is a direct result of the Parareal algorithm and the sources of error present in it. The Parareal method consists of an initial approximation to the solution performed by the coarse solver which is accurate to within the coarse error predicted by \cref{thm:coarse}. This is followed by a series of parallel-in-time corrections which converge to the full solution, derived from the difference between the coarse and fine solutions. The closer the initial approximation is to the solution, the less correction is required to converge.

The optimal choice of $\eta$ may be written as an optimisation problem:

\begin{equation}
    \min_{\eta\in\mathbb{R}^{+}}\left(C_{1}\Delta T^{3}\varepsilon\Lambda(\eta) + (C_{2} + C_{3}\varepsilon)\varepsilon\eta\right),
    \label{eq:optim}
\end{equation}
for some as-yet unknown constants $C_{1}$, $C_{2}$, and $C_{3}$. It is here that the fact that both the timestepping and averaging errors are bounded proportionally to the norm of the nonlinear term, $M$, becomes serendipitous, as this constant may be somewhat non-optimal in practice. In seeking the location, but not the magnitude, of the minimum coarse error, the bound on the norm of the nonlinear operator plays no role. Seeking stationary points with respect to $\eta$, this then requires us to find $\eta$ such that:

\begin{equation}
    \frac{\mathrm{d}}{\mathrm{d}\eta}\max_{\epsilon_{\beta}}\max_{\mathcal{S}_{k,\alpha}^{\epsilon_{\beta}}}\frac{\left|\omega_{k}^{\alpha}-\omega_{k_{1}}^{\alpha_{1}} - \omega_{k_{2}}^{\alpha_{2}}\right|^{p}}{\varepsilon}\int_{0}^{1}\rho(s)e^{\frac{i\left|\omega_{k}^{\alpha}-\omega_{k_{1}}^{\alpha_{1}} - \omega_{k_{2}}^{\alpha_{2}}\right|\eta\Delta T}{\varepsilon}s}\,\mathrm{d}s + \frac{C_{2} + C_{3}\varepsilon}{C_{1}\Delta T^{3}} = 0.
    \label{eq:opt_diff}
\end{equation}

The result of~\cref{eq:opt_diff} is used to choose the optimal averaging window. This result captures the relationship of parameters such as the timestep and the scale separation on the optimal averaging, but relies on several unknown constants. If these constants $C_{n}$ were known, the optimal averaging window could be determined computationally. \Cref{eq:opt_diff} would then provide an approximation to the optimal window. Given some initial data of the type shown in~\cref{fig:apint_conv}, these constants may be fit by least-squares. Doing so fits the known trend to the known data, and permits the optimal averaging window to be recomputed `on the fly' in a computation.

Certain practical issues arise in the computation of $\eta$. Firstly, the computation of $\frac{\mathrm{d}\Lambda}{\mathrm{d}\eta}$ requires all triads to be investigated, i.e. the maximum is taken over the set of all near-resonant sets. Doing so is computationally expensive, although if this computation were to be performed infrequently the cost could be negligible compared to the simulation cost. Additionally, finding $\eta$ requires solving a transcendental equation in at least two variables ($\eta$, $\varepsilon$), both for the initial fitting of constants, and for the optimisation on the fly. We therefore propose a simpler model based on the behaviour of $\Lambda(s)$.

Restricting ourselves for this example to a Gaussian kernel, we may consider the asymptotic behaviour of the kernel as $\lambda_{n}$ is large. This gives:

\begin{equation}
    \frac{1}{\eta}\int_{0}^{\eta}\rho\left(\frac{s}{\eta}\right)e^{i\lambda_{n}\Delta Ts}ds=\int_{0}^{1}\rho\left(s\right)e^{i\lambda_{n}\Delta T\eta s}ds\sim C_{0}e^{-C_{1}(\left|\lambda_{n}\right|\Delta T\eta)^{2}}.
    \label{eq:gauss_int}
\end{equation}

We then multiply our approximation by $1=\eta^{2}/\eta^{2}$, to obtain:

\begin{equation}
    \frac{\eta^{2}}{\eta^{2}}C_{1}\Delta T^{3}\varepsilon\Lambda(\eta) \approx \frac{D_{1}\Delta T\varepsilon}{\eta^{2}},
\end{equation}

\begin{figure}[H]
    \centering
    \includegraphics[scale=0.7]{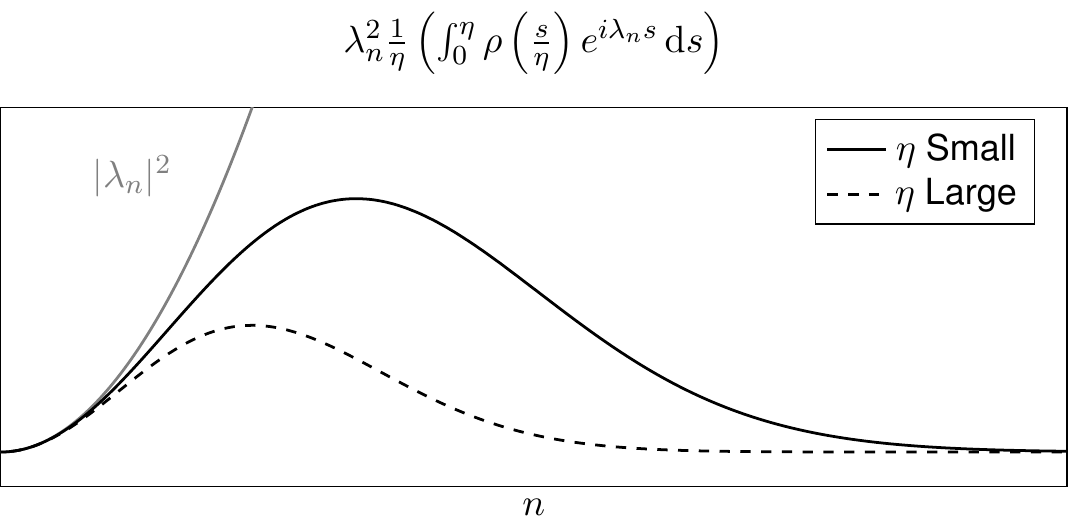}
    \caption{Examples of the function $x^{2}e^{-Cx^{2}}$, showing that it is bounded independently of $x$ and that it tends rapidly to zero as $x\to\infty$. This is used conceptually in bounding the $\Lambda$-term, i.e. the mitigated stiffness, in \cref{eq:gauss_int}. The unmitigated stiffness is shown in grey.}
    \label{fig:bounds_function}
\end{figure}

for some constant, $D_{1}$, since $x^{2}e^{-x^{2}}$ is bounded independently of $x$ (\textit{cf.}~\ref{fig:bounds_function}). We then replace our first term in~\cref{eq:optim} and seek fixed points corresponding to the minimum error. This yields:

\begin{equation}
    \eta_{\mathrm{optimal}} = \sqrt{\frac{D_{1}\Delta T}{C_{2} + C_{3}\varepsilon}}.
    \label{eq:phi_min}
\end{equation}

This equation provides an estimate for the optimal averaging window length, $\eta_{\mathrm{opt}}$, in terms of the computational parameters and the empirically-fit constants. This result is consistent with \cref{thm:apint_conv}, as it exhibits a clear minimum for $\mathcal{O}(1)$ values of $\varepsilon$, with the optimal averaging window increasing as $\varepsilon \to 0$, as the asymptotic theory predicts. Both approximations are shown in~\cref{fig:window_prediction} for a set of minima extracted from a series of runs of the algorithm.

The full model given in~\cref{eq:opt_diff} provides a much closer approximation both to the behaviour for $\varepsilon = 1$ and as $\varepsilon\to 0$, and as an actual fit to the points. It does this, however, at the cost of several orders of magnitude more computational difficulty. The simple model of~\cref{eq:phi_min}, on the other hand, provides a reasonable approximation to the error as a function of $\varepsilon$, but has the disadvantage of poorly resolving the trend in the limit as $\varepsilon\to 0$. While the simple prediction underestimates the optimal as $\varepsilon\to 0$, the behaviour in this range is well-understood (\textit{cf.} \cite{Haut_Wingate_14}, \cite{Ariel_etal_16}) and so a hybrid model may easily be applied in practice.

\begin{figure}[H]
\centering
    \includegraphics[scale=0.55]{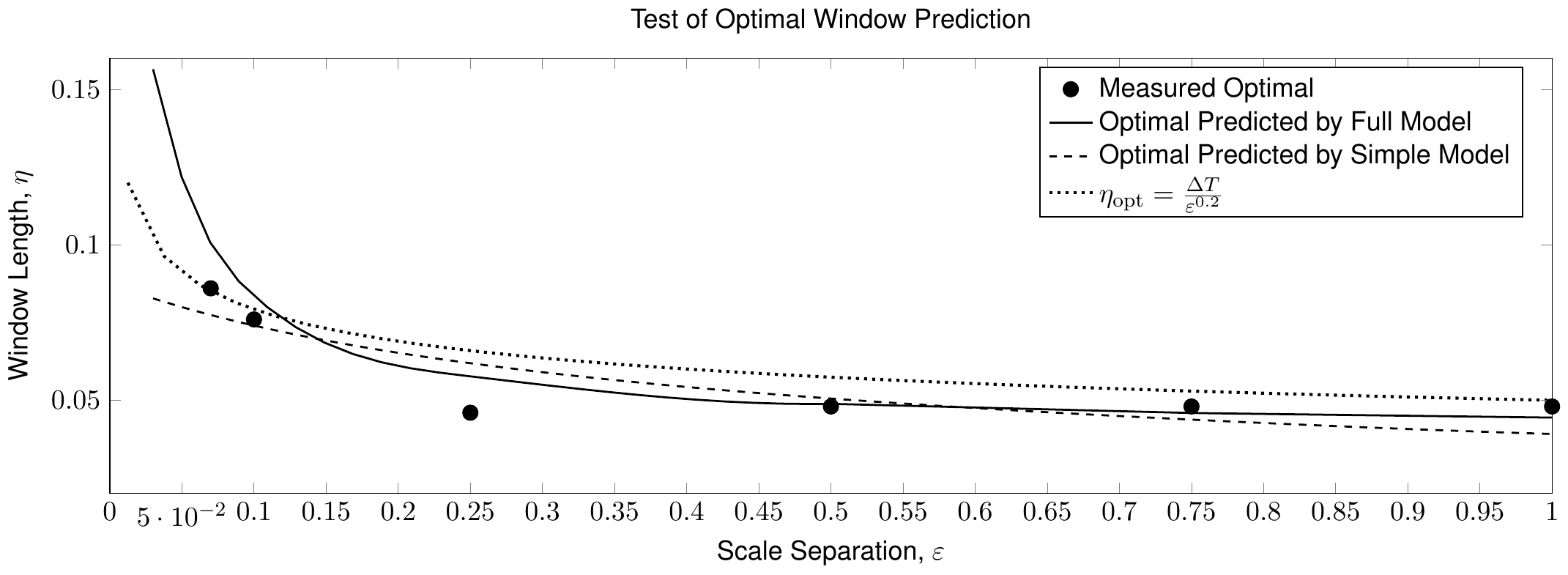}
    \caption{The optimal averaging window for the APinT solver is predicted in three different ways as a function of the scale separation. The measured optimal values for several runs with a coarse timestep of $\Delta T = 0.05$ are shown with the filled circles. The so-called `full' or expensive model from \cref{eq:opt_diff} is shown with the green curve. This mode shows good agreement throughout the range and handles the long averaging windows needed as $\varepsilon\to 0$ as well. The simple model derived from asymptotic analysis on a Gaussian kernel is shown in red, and provides similar accuracy as the full model outside of the small-$\varepsilon$ region, but at a dramatically reduced computational cost. Finally, the dashed line indicated the assumed scaling on the averaging window given in \cref{eq:eta_scaling} with $s$ taken empirically as 0.2. The trend towards a longer time averaging window being necessary for smaller $\varepsilon$ is captured, while this scaling somehwat overestimates the window for larger values of scale separation, although it may be computed very cheaply.}
    \label{fig:window_prediction}
\end{figure}

\section{Conclusion}

We have investigated the convergence of a Parareal method using the APinT coarse solver, which provides a technique by which oscillatory-stiff equations may be solved with the Parareal method. The convergence of this method is due to the averaging applied to the coarse solution, which filters the fast waves and mitigates the oscillatory stiffness present in many of the equations of mathematical physics. This averaging must be performed over the entirety of the nonlinear operator due to the role the direct and near-resonances play in the oscillatory stiffness of the system.

By describing the error of the coarse solver in terms of the interplay between the average over the rapid oscillations and the timestepping, we show the method converges for finite scale separation, significantly extending the domain of applicability for this method.

We have shown here that this method is convergent across a wide range of scale separation, which is an improvement on the prior result\cite{Haut_Wingate_14} which held only in the small-$\varepsilon$ limit. Further, in \cref{sect:optimal_averaging} we considered both a full and a reduced model to predict the optimal averaging window in practical codes. 

\appendix
\section{Proof of \cref{lemma:avg_err}}
\label{sect:avg_err}

Consider 
\[
\frac{d\mathbf{u}}{dt}\left(t\right)=\mathbf{f}\left(\frac{t}{\varepsilon},\mathbf{u}\left(t\right)\right),\,\,\,\,0\leq t\leq h,
\]
and its averaged version 
\[
\frac{d\overline{\mathbf{u}}}{dt}\left(t\right)=\overline{\mathbf{f}}_{\eta'}\left(\frac{t}{\varepsilon},\overline{\mathbf{u}}\left(t\right)\right),\,\,\,\,0\leq t\leq h,
\]
where with $\eta'=\varepsilon h\eta$,
\[
\overline{\mathbf{f}}_{\eta'}\left(\frac{t}{\varepsilon},\overline{\mathbf{u}}\left(t\right)\right)=\frac{1}{\eta'}\int_{0}^{\eta'}\rho\left(\frac{s}{\eta'}\right)\mathbf{f}\left(\frac{t+s}{\varepsilon},\overline{\mathbf{u}}\left(t\right)\right)ds.
\]
Then 
\[
\left\Vert \mathbf{u}\left(t\right)-\overline{\mathbf{u}}\left(t\right)\right\Vert =\mathcal{O}\left(\varepsilon\eta h\right),\,\,\,\,0\leq t\leq h.
\]

To prove this, change variables: $\tau=t/\left(h\varepsilon\right)$
and 
\[
\mathbf{v}\left(\tau\right)=\mathbf{v}\left(t/h\varepsilon\right)=\mathbf{u}\left(t\right),
\]
\[
\overline{\mathbf{v}}\left(\tau\right)=\overline{\mathbf{v}}\left(t/h\varepsilon\right)=\overline{\mathbf{u}}\left(t\right).
\]
Then 
\begin{eqnarray*}
\frac{d}{dt}\mathbf{u}\left(t\right) & = & \frac{d}{dt}\mathbf{v}\left(t/h\varepsilon\right)\\
 & = & \frac{1}{h\varepsilon}\frac{d\mathbf{v}}{d\tau}\left(\tau\right).
\end{eqnarray*}
Thus,
\[
\frac{d\mathbf{v}}{d\tau}\left(\tau\right)=h\varepsilon\mathbf{f}\left(h\tau,\mathbf{v}\left(\tau\right)\right)\equiv h\varepsilon\mathbf{g}_{h}\left(\tau,\mathbf{v}\left(\tau\right)\right),\,\,\,\,0\leq\tau\leq\frac{1}{\varepsilon},
\]
and, with $h\tau'=s/\varepsilon$, $ds=\varepsilon hd\tau'$,
\begin{eqnarray*}
\overline{\mathbf{f}}_{\eta}\left(\frac{t}{\varepsilon},\overline{\mathbf{v}}\left(\tau\right)\right) & = & \frac{1}{\varepsilon h\eta}\int_{0}^{\varepsilon h\eta}\rho\left(\frac{s}{\varepsilon h\eta}\right)\mathbf{f}\left(\frac{t+s}{\varepsilon},\overline{\mathbf{v}}\left(\tau\right)\right)ds\\
 & = & \frac{1}{\varepsilon h\eta}\int_{0}^{\varepsilon h\eta}\rho\left(\frac{s}{\varepsilon h\eta}\right)\mathbf{f}\left(\frac{s}{\varepsilon}+h\tau,\overline{\mathbf{v}}\left(\tau\right)\right)ds\\
 & = & \frac{\varepsilon h}{\varepsilon h\eta}\int_{0}^{\varepsilon h\eta/\left(\varepsilon h\right)}\rho\left(\frac{\varepsilon h\tau'}{\varepsilon h\eta}\right)\mathbf{f}\left(h\tau'+h\tau,\overline{\mathbf{v}}\left(\tau\right)\right)d\tau'\\
 & = & \frac{1}{\eta}\int_{0}^{\eta}\rho\left(\frac{\tau'}{\eta}\right)\mathbf{f}\left(h\left(\tau'+\tau\right),\overline{\mathbf{v}}\left(\tau\right)\right)d\tau'\\
 & = & \frac{1}{\eta}\int_{0}^{\eta}\mathbf{g}_{h}\left(\tau,+\tau',\mathbf{v}\left(\tau\right)\right)d\tau'\\
 & = & \overline{\left(\mathbf{g}_{h}\right)}_{\eta}\left(\tau,\mathbf{v}\left(\tau\right)\right).
\end{eqnarray*}
Thus, 
\[
\frac{d\overline{\mathbf{v}}}{d\tau}\left(\tau\right)=h\varepsilon\frac{d\overline{\mathbf{u}}}{dt}\left(t\right)=h\varepsilon\overline{\mathbf{f}}_{\eta}\left(\frac{t}{\varepsilon},\overline{\mathbf{v}}\left(\tau\right)\right)=h\varepsilon\overline{\left(\mathbf{g}_{h}\right)}_{\eta}\left(\tau,\mathbf{v}\left(\tau\right)\right),\,\,\,\,0\leq\tau\leq\frac{1}{\varepsilon}.
\]

Define 
\[
\mathbf{f}_{\eta}\left(t,\mathbf{x}\right)=\frac{1}{\eta}\int_{0}^{\eta}\rho\left(\frac{s}{\eta}\right)\mathbf{f}\left(t+s,\mathbf{x}\right)ds.
\]

\begin{lemma}
If $\phi\left(t\right)$ is Lipschitz-continuous with Lipschitz-constant
$\lambda$. Then
\[
\left|\phi\left(t\right)-\phi_{\eta}\left(t\right)\right|\leq C_{0}\lambda\eta,
\]
where
\[
C_{0}=\int_{0}^{1}\rho\left(s\right)sds.
\]
\label{lemma:pt1}
\end{lemma}
\begin{proof}
Using that
\[
\frac{1}{\eta}\int_{0}^{\eta}\rho\left(\frac{s}{\eta}\right)ds=\int_{0}^{1}\rho\left(s\right)ds=1,
\]
we have that 
\begin{eqnarray*}
\left|\phi\left(t\right)-\phi_{\eta}\left(t\right)\right| & = & \left|\phi\left(t\right)-\frac{1}{\eta}\int_{0}^{\eta}\rho\left(\frac{s}{\eta}\right)\phi\left(s+t\right)ds\right|\\
 & = & \frac{1}{\eta}\int_{0}^{\eta}\rho\left(\frac{s}{\eta}\right)\left|\phi\left(t\right)-\phi\left(s+t\right)\right|ds\\
 & \leq & \frac{1}{\eta}\int_{0}^{\eta}\rho\left(\frac{s}{\eta}\right)s\lambda ds\\
 & = & \eta\lambda\int_{0}^{1}\rho\left(s\right)sds.
\end{eqnarray*}

\end{proof}
\begin{lemma}
    \label{lemma:pt2}
Consider 
\[
\frac{d\mathbf{v}}{d\tau}\left(t\right)=h\varepsilon\mathbf{f}\left(ht,\mathbf{v}\left(t\right)\right),\,\,\,0\leq t\leq\varepsilon^{-1},
\]

\end{lemma}
with $\mathbf{f}$ continuous in each argument. Also assume that 
\[
\left\Vert \mathbf{f}\left(ht,\mathbf{u}\right)-\mathbf{f}\left(ht,\mathbf{w}\right)\right\Vert \leq\lambda\left\Vert \mathbf{u}-\mathbf{w}\right\Vert ,
\]
and 
\[
M=\sup_{x\in D}\sup_{0\leq t\leq\varepsilon^{-1}}\left\Vert \mathbf{f}\left(ht,\mathbf{w}\right)\right\Vert <\infty.
\]
Then defining 
\[
\phi\left(t\right)=\int_{0}^{t}\mathbf{f}\left(h\tau,\mathbf{v}\left(\tau\right)\right)d\tau,
\]
we have that 
\[
\left|\phi_{\eta}\left(t\right)-\int_{0}^{t}\mathbf{f}_{\eta}\left(h\tau,\mathbf{v}\left(\tau\right)\right)d\tau\right|\leq C_{0}\left(1+\lambda h\right)M\eta.
\]

\begin{proof}
We calculate that 
\begin{eqnarray*}
\phi_{\eta}\left(t\right) & = & \frac{1}{\eta}\int_{0}^{\eta}\rho\left(\frac{s}{\eta}\right)\phi\left(s+t\right)ds\\
 & = & \frac{1}{\eta}\int_{0}^{\eta}\rho\left(\frac{s}{\eta}\right)\left(\int_{0}^{t+s}\mathbf{f}\left(h\tau,\mathbf{v}\left(\tau\right)\right)d\tau\right)ds\\
 & = & \frac{1}{\eta}\int_{0}^{\eta}\rho\left(\frac{s}{\eta}\right)\left(\int_{s}^{t+s}\mathbf{f}\left(h\tau,\mathbf{v}\left(\tau\right)\right)d\tau\right)ds+R_{1}\\
 & = & \frac{1}{\eta}\int_{0}^{\eta}\rho\left(\frac{s}{\eta}\right)\left(\int_{0}^{t}\mathbf{f}\left(h\left(\tau+s\right),\mathbf{v}\left(\tau+s\right)\right)d\tau\right)ds+R_{1}\\
 & = & \frac{1}{\eta}\int_{0}^{\eta}\rho\left(\frac{s}{\eta}\right)\left(\int_{0}^{t}\mathbf{f}\left(h\left(\tau+s\right),\mathbf{v}\left(\tau\right)\right)d\tau\right)ds+R_{1}+R_{2}\\
 & = & \int_{0}^{t}\left(\frac{1}{\eta}\int_{0}^{\eta}\rho\left(\frac{s}{\eta}\right)\mathbf{f}\left(h\left(\tau+s\right),\mathbf{v}\left(\tau\right)\right)ds\right)d\tau+R_{1}+R_{2}\\
 & = & \int_{0}^{t}\int_{0}^{t}\mathbf{f}_{\eta}\left(h\tau,\mathbf{v}\left(\tau\right)\right)d\tau+R_{1}+R_{2},
\end{eqnarray*}
where 
\begin{eqnarray*}
\left\Vert R_{1}\right\Vert  & = & \left\Vert \frac{1}{\eta}\int_{0}^{\eta}\rho\left(\frac{s}{\eta}\right)\left(\int_{0}^{s}\mathbf{f}\left(h\tau,\mathbf{v}\left(\tau\right)\right)d\tau\right)ds\right\Vert \\
 & \leq & \frac{1}{\eta}\int_{0}^{\eta}\rho\left(\frac{s}{\eta}\right)\int_{0}^{s}\left\Vert \mathbf{f}\left(h\tau,\mathbf{v}\left(\tau\right)\right)\right\Vert d\tau ds\\
 & \leq & \frac{1}{\eta}\int_{0}^{\eta}\rho\left(\frac{s}{\eta}\right)\int_{0}^{s}Md\tau ds\\
 & = & M\frac{1}{\eta}\int_{0}^{\eta}\rho\left(\frac{s}{\eta}\right)sds\\
 & = & M\eta\int_{0}^{1}\rho\left(s\right)sds\\
 & = & C_{0}M\eta,
\end{eqnarray*}
and 
\begin{eqnarray*}
\left\Vert R_{2}\right\Vert  & = & \left\Vert \frac{1}{\eta}\int_{0}^{\eta}\rho\left(\frac{s}{\eta}\right)\int_{0}^{t}\left(\mathbf{f}\left(h\left(\tau+s\right),\mathbf{v}\left(\tau+s\right)\right)-\mathbf{f}\left(h\left(\tau+s\right),\mathbf{v}\left(\tau\right)\right)\right)d\tau ds\right\Vert \\
 & \leq & \frac{1}{\eta}\int_{0}^{\eta}\rho\left(\frac{s}{\eta}\right)\int_{0}^{t}\left\Vert \mathbf{f}\left(h\left(\tau+s\right),\mathbf{v}\left(\tau+s\right)\right)-\mathbf{f}\left(h\left(\tau+s\right),\mathbf{v}\left(\tau\right)\right)\right\Vert d\tau ds\\
 & \leq & \frac{1}{\eta}\lambda\int_{0}^{\eta}\rho\left(\frac{s}{\eta}\right)\int_{0}^{t}\left\Vert \mathbf{v}\left(\tau+s\right)-\mathbf{v}\left(\tau\right)\right\Vert d\tau ds\\
 & = & \frac{1}{\eta}\lambda\int_{0}^{\eta}\rho\left(\frac{s}{\eta}\right)\int_{0}^{t}\left\Vert \int_{\tau}^{s+\tau}\frac{d\mathbf{v}}{d\sigma}\left(\sigma\right)d\sigma\right\Vert d\tau ds\\
 & = & \frac{1}{\eta}\lambda\int_{0}^{\eta}\rho\left(\frac{s}{\eta}\right)\int_{0}^{t}\left\Vert \int_{\tau}^{s+\tau}h\varepsilon\mathbf{f}\left(h\sigma,\mathbf{v}\left(\sigma\right)\right)d\sigma\right\Vert d\tau ds\\
 & \leq & \frac{1}{\eta}h\varepsilon\lambda\int_{0}^{\eta}\rho\left(\frac{s}{\eta}\right)\int_{0}^{t}\int_{\tau}^{s+\tau}\left\Vert \mathbf{f}\left(h\sigma,\mathbf{v}\left(\sigma\right)\right)\right\Vert d\sigma d\tau ds\\
 & \leq & \frac{1}{\eta}h\varepsilon\lambda M\int_{0}^{\eta}\rho\left(\frac{s}{\eta}\right)\int_{0}^{t}\int_{\tau}^{s+\tau}d\sigma d\tau ds\\
 & = & \frac{1}{\eta}h\varepsilon\lambda M\int_{0}^{\eta}\rho\left(\frac{s}{\eta}\right)\int_{0}^{t}sd\tau ds\\
 & = & C_{0}\eta h\lambda M\varepsilon t\\
 & \leq & C_{0}h\eta\lambda M.
\end{eqnarray*}
In the last inequality, we used that $0\leq t\leq\varepsilon^{-1}$.\end{proof}
Consider 
\[
\frac{d\mathbf{v}}{d\tau}\left(t\right)=h\varepsilon\mathbf{f}\left(ht,\mathbf{v}\left(t\right)\right),\,\,\,0\leq t\leq\varepsilon^{-1},
\]
with the same assumptions as in the previous lemmas. Let 
\[
\frac{d\mathbf{\overline{v}}}{d\tau}\left(t\right)=h\varepsilon\mathbf{f}_{\eta}\left(ht,\overline{\mathbf{v}}\left(t\right)\right),\,\,\,0\leq t\leq\varepsilon^{-1}.
\]
Then 
\[
\left\Vert \mathbf{v}\left(t\right)-\overline{\mathbf{v}}\left(t\right)\right\Vert \leq C_{1}h\varepsilon\eta,\,\,\,\,0\leq ht\leq\varepsilon^{-1}.
\]
Note that 
\[
\mathbf{v}\left(t\right)=\mathbf{v}\left(0\right)+h\varepsilon\int_{0}^{t}\mathbf{f}\left(h\tau,\mathbf{v}\left(\tau\right)\right)d\tau.
\]
By \cref{lemma:pt2},
\[
\int_{0}^{t}\mathbf{f}\left(h\tau,\mathbf{v}\left(\tau\right)\right)d\tau=\int_{0}^{t}\mathbf{f}_{\eta}\left(h\tau,\mathbf{v}\left(\tau\right)\right)d\tau+E_{0},
\]
where 
\[
\left\Vert E_{0}\right\Vert \leq C_{0}\left(1+\lambda h\right)M\eta.
\]
Therefore,
\[
\mathbf{v}\left(t\right)=\mathbf{v}\left(0\right)+h\varepsilon\int_{0}^{t}\mathbf{f}_{\eta}\left(h\tau,\mathbf{v}\left(\tau\right)\right)d\tau+E_{1},
\]
where 
\[
\left\Vert E_{1}\right\Vert =\left\Vert h\varepsilon E_{0}\right\Vert \leq C_{0}\left(1+\lambda h\right)M\eta h\varepsilon.
\]
Also, since 
\[
\overline{\mathbf{v}}\left(t\right)=\mathbf{v}\left(0\right)+h\varepsilon\int_{0}^{t}\mathbf{f}_{\eta}\left(h\tau,\overline{\mathbf{v}}\left(t\right)\right)d\tau,
\]
we have that 
\begin{eqnarray*}
\left\Vert \mathbf{v}\left(t\right)-\overline{\mathbf{v}}\left(t\right)\right\Vert  & \leq & h\varepsilon\int_{0}^{t}\left\Vert \mathbf{f}_{\eta}\left(h\tau,\mathbf{v}\left(\tau\right)\right)-\mathbf{f}_{\eta}\left(h\tau,\overline{\mathbf{v}}\left(t\right)\right)\right\Vert d\tau+C_{0}\left(1+\lambda h\right)M\eta h\varepsilon\\
 & \leq & h\varepsilon\lambda\int_{0}^{t}\left\Vert \mathbf{v}\left(\tau\right)-\overline{\mathbf{v}}\left(t\right)\right\Vert d\tau+C_{0}\left(1+\lambda h\right)M\eta h\varepsilon.
\end{eqnarray*}
Finally, by Gronwall's inequality, 
\[
\left\Vert \mathbf{v}\left(t\right)-\overline{\mathbf{v}}\left(t\right)\right\Vert \leq C_{0}\left(1+\lambda h\right)M\eta h\varepsilon e^{h\varepsilon\lambda t}.
\]

\newpage
\section*{Acknowledgments}
We would like the acknowledge the support of the University of Exeter and Los Alamos National Laboratory.

\bibliographystyle{siamplain}
\bibliography{references}

\begin{thebibliography}{10}

\bibitem{Ariel_etal_16}
{\sc G.~Ariel, S.~J. Kim, and R.~Tsai}, {\em Parareal methods for highly
  oscillatory dynamical systems}, SIAM Journal on Scientific Computing,
  (2016).

\bibitem{Audozze_etal_09}
{\sc C.~Audouze, M.~Massot, and S.~Volz}, {\em {Symplectic multi-time step
  parareal algorithms applied to molecular dynamics}}.
\newblock Submitted to SIAM Journal of Scientific Computing, 2009.

\bibitem{Baffico_etal_02}
{\sc L.~Baffico, S.~Bernard, Y.~Maday, G.~Turinici, and G.~Z\'erah}, {\em
  Parallel-in-time molecular-dynamics simulations}, Phys. Rev. E, 66 (2002).

\bibitem{Bal_05}
{\sc G.~Bal}, {\em On the Convergence and the Stability of the Parareal
  Algorithm to Solve Partial Differential Equations}, Springer Berlin
  Heidelberg, Berlin, Heidelberg, 2005, pp.~425--432.

\bibitem{Bal_Maday_02}
{\sc G.~Bal and Y.~Maday}, {\em A parareal time discretization for non-linear
  pde’s with application to the pricing of an american put}, in Recent
  Developments in Domain Decomposition Methods, L.~F. Pavarino and A.~Toselli,
  eds., vol.~23 of Lecture Notes in Computational Science and Engineering,
  Springer Berlin Heidelberg, 2002, pp.~189--202.

\bibitem{Bal_Wu_08}
{\sc G.~Bal and Q.~Wu}, {\em Symplectic Parareal}, Springer Berlin Heidelberg,
  Berlin, Heidelberg, 2008, pp.~401--408.

\bibitem{Bogoliubov_61}
{\sc N.~Bogoliubov and Y.~Mitropolsky}, {\em Asymptotic Methods in the Theory
  of Nonlinear Oscillations}, Gordon and Breach, New York, 1961.

\bibitem{Charney_48}
{\sc J.~G. Charney}, {\em On the scale of atmospheric motions}, Geophysiske
  Publikasjoner, 17 (1948), pp.~3--17.

\bibitem{Charney_49}
{\sc J.~G. Charney}, {\em On a physical basis for numerical prediction of
  large-scale motions in the atmosphere}, Journal of Meteorology, 6 (1949),
  pp.~371--385.

\bibitem{Charney_Phillips_53}
{\sc J.~G. Charney and N.~Phillips}, {\em A numerical integration of the
  quasigeostrophic equations for barotropic and simple baroclinic flows},
  Journal of Meteorology, 10 (1953), pp.~71--99.

\bibitem{Davies_etal_03}
{\sc T.~Davies, A.~Staniforth, N.~Wood, and J.~Thuburn}, {\em Validity of
  anelastic and other equation sets as inferred from normal-mode analysis},
  Quarterly Journal of the Royal Meteorological Society, 129 (2003),
  pp.~2761--2775.

\bibitem{E_03}
{\sc W.~E}, {\em Analysis of the heterogeneous multiscale method for ordinary
  differential equations}, Commun. Math. Sci., 1 (2003), pp.~423--436.

\bibitem{Embid_Majda_96}
{\sc P.~F. Embid and A.~J. Majda}, {\em Averaging over fast gravity waves for
  geophysical flows with arbitrary potential vorticity}, Communications in
  Partial Differential Equations, 21 (1996), pp.~619--658.

\bibitem{Engquist_Tsai_05}
{\sc B.~Engquist and Y.-H. Tsai}, {\em Heterogeneous multiscale methods for
  stiff ordinary differential equations}, Mathematics of Computation, 74
  (2005), pp.~1707--1742.

\bibitem{Farhat_Chandesris_03}
{\sc C.~Farhat and M.~Chandesris}, {\em Time-decomposed parallel
  time-integrators: theory and feasibility studies for fluid, structure, and
  fluid–structure applications}, International Journal for Numerical Methods
  in Engineering, 58 (2003), pp.~1397--1434.

\bibitem{Fischer_etal_03}
{\sc P.~F. Fischer, F.~Hecht, and Y.~Maday}, {\em A parareal in time
  semi-implicit approximation of the navier-stokes equations}, in Domain
  Decomposition Methods in Science and Engineering, T.~J. Barth, M.~Griebel,
  D.~E. Keyes, R.~M. Nieminen, D.~Roose, T.~Schlick, R.~Kornhuber, R.~Hoppe,
  J.~Périaux, O.~Pironneau, O.~Widlund, and J.~Xu, eds., vol.~40 of Lecture
  Notes in Computational Science and Engineering, Springer Berlin Heidelberg,
  2005, pp.~433--440.

\bibitem{Kramer_etal_02}
{\sc The Fourth International Conference on Dynamical Systems and Differential
  Equations}, {\em Application of Weak Turbulence Theory to {FPU} Model}, 2002.

\bibitem{Gander_15}
{\sc M.~J. Gander}, {\em 50 years of time parallel time integration}, in
  Multiple Shooting and Time Domain Decomposition, T.~Carraro, M.~Geiger,
  S.~Korkel, and R.~Rannacher, eds., Springer-Verlag, 2015.

\bibitem{Gander_Hairer_14}
{\sc M.~J. Gander and E.~Hairer}, {\em Analysis for parareal algorithms applied
  to hamiltonian differential equations}, J. Comput. Appl. Math., 259 (2014),
  pp.~2--13.

\bibitem{Garrido_etal_03}
{\sc I.~Garrido, M.~S. Espedal, and G.~E. Fladmark}, {\em A convergent
  algorithm for time parallelization applied to reservoir simulation}, in
  Domain Decomposition Methods in Science and Engineering, T.~J. Barth,
  M.~Griebel, D.~E. Keyes, R.~M. Nieminen, D.~Roose, T.~Schlick, R.~Kornhuber,
  R.~Hoppe, J.~Périaux, O.~Pironneau, O.~Widlund, and J.~Xu, eds., vol.~40 of
  Lecture Notes in Computational Science and Engineering, Springer Berlin
  Heidelberg, 2005, pp.~469--476.

\bibitem{Haut_Wingate_14}
{\sc T.~S. Haut and B.~A. Wingate}, {\em An asymptotic parallel-in-time method
  for highly oscillatory pdes}, SIAM Journal on Scientific Computing, 36
  (2014), pp.~A693--A713.

\bibitem{He_10}
{\sc L.~He}, {\em The reduced basis technique as a coarse solver for parareal
  in time simulations.}, J. Comput. Math., 28 (2010), pp.~676 -- 692.

\bibitem{Higham_93}
{\sc D.~J. Higham and L.~N. Trefethen}, {\em Stiffness of odes}, BIT Numerical
  Mathematics, 33 (1993), pp.~285--303.

\bibitem{Kadri_16}
{\sc U.~Kadri and T.~R. Akylas}, {\em On resonant triad interactions of
  acoustic-gravity waves}, J. Fluid Mech., 788 (2016).

\bibitem{Kincaid_Cheney_91}
{\sc D.~Kincaid and W.~Cheney}, {\em Numerical Analysis: Mathematics of
  Scientific Computing}, Brooks/Cole Publishing Co., Pacific Grove, CA, USA,
  1991.

\bibitem{Klainerman_Majda_81}
{\sc S.~Klainerman and A.~J. Majda}, {\em Singular limits of quasilinear
  hyperbolic systems with large parameters and the incompressible limit of
  compressible fluids}, Communications in Pure and Applied Mathematics, 34
  (1981), pp.~481--524.

\bibitem{Legoll_etal_13}
{\sc F.~Legoll, T.~Lelièvre, and G.~Samaey}, {\em A micro-macro parareal
  algorithm: Application to singularly perturbed ordinary differential
  equations.}, SIAM J. Scientific Computing, 35 (2013).

\bibitem{Lions_etal_01}
{\sc J.~Lions, Y.~Maday, and G.~Turinici}, {\em A''parareal''in time
  discretization of pde's}, Comptes Rendus de l'Academie des Sciences Series I
  Mathematics, 332 (2001), pp.~661--668.

\bibitem{Maday_Turinici_03}
{\sc Y.~Maday and G.~Turinici}, {\em Parallel in time algorithms for quantum
  control: Parareal time discretization scheme}, International journal of
  quantum chemistry, 93 (2003), pp.~223--228.

\bibitem{Maday_Turinici_05}
{\sc Y.~Maday and G.~Turinici}, {\em The parareal in time iterative solver: a
  further direction to parallel implementation}, in Domain decomposition
  methods in science and engineering, Springer Berlin Heidelberg, 2005,
  pp.~441--448.

\bibitem{Majda_Book}
{\sc A.~Majda}, {\em Introduction to PDEs and Waves for the Atmosphere and
  Ocean: Courant Lecture Notes Vol. 9}, American Mathematical Society and
  Courant Institute of Mathematical Sciences, 2002.

\bibitem{Sanders_Verhulst}
{\sc J.~A. Sanders, F.~Verhulst, and J.~Murdock}, {\em Averaging methods in
  nonlinear dynamical systems}, Applied mathematical sciences, Springer, New
  York, Berlin, Heidelberg, 2~ed., 2007.

\bibitem{Schochet_94}
{\sc S.~Schochet}, {\em Fast singular limits of hyperbolic pdes}, Journal of
  Differential Equations, 114 (1994), pp.~476--512.

\bibitem{Vallis_06}
{\sc G.~K. Vallis}, {\em Atmospheric and Oceanic Fluid Dynamics}, Cambridge
  University Press, Cambridge, U.K., 2006.

\bibitem{Ward_Dewar_10}
{\sc M.~L. Ward and W.~K. Dewar}, {\em Scattering of gravity waves by potential
  vorticity in a shallow-water fluid}, Journal of Fluid Mechanics, 663 (2010),
  pp.~478--506.

\end{thebibliography}
\end{document}